\documentclass[12pt,reqno]{amsart}
\usepackage[margin=1in]{geometry}
\usepackage{latexsym}
\usepackage{amsfonts}
\usepackage{amsmath}
\usepackage{amssymb}
\usepackage{amsthm}
\usepackage{enumerate}
\usepackage[hang,flushmargin]{footmisc}
\usepackage{caption}
\usepackage{tabu}
\usepackage{mathrsfs}
\usepackage{amsaddr}
\usepackage{subfig}

\newcommand{\Av}{\operatorname{Av}}
\DeclareMathOperator{\DPT}{\mathsf{DPT}}

\DeclareMathOperator{\VHC}{\mathsf{VHC}}
\newcommand{\swu}{\operatorname{swu}}
\newcommand{\swd}{\operatorname{swd}}
\newcommand{\swl}{\operatorname{swl}}
\newcommand{\swr}{\operatorname{swr}}

\usepackage{graphicx}
\usepackage{epstopdf}
\usepackage{epsfig}
\usepackage{caption}
 
\usepackage{bm} 
\usepackage{cite}

\makeatletter
\newtheorem*{rep@theorem}{\rep@title}
\newcommand{\newreptheorem}[2]{%
\newenvironment{rep#1}[1]{%
 \def\rep@title{#2 \ref{##1}}%
 \begin{rep@theorem}}%
 {\end{rep@theorem}}}
\makeatother

\newtheorem{theorem}{Theorem}[section]
\newreptheorem{theorem}{Conjecture}
\newtheorem{lemma}{Lemma}[section]
\newtheorem{proposition}{Proposition}[section]
\newtheorem{corollary}{Corollary}[section]

\newtheorem{question}{Question}[section]
\newtheorem{problem}{Problem}[section]

\theoremstyle{definition}
\newtheorem{definition}{Definition}[section]
\newtheorem{remark}{Remark}[section]

\DeclareMathOperator{\Id}{id}

\DeclareMathOperator{\Des}{Des}
\DeclareMathOperator{\des}{des}

\DeclareMathOperator{\peak}{peak}
\DeclareMathOperator{\rmax}{rmax}
\DeclareMathOperator{\tl}{tl}

\DeclareMathOperator{\zeil}{zeil}
\DeclareMathOperator{\rot}{rot}

\begin{document}
\title{Fertility, Strong Fertility, and Postorder Wilf Equivalence}
\author{Colin Defant}
\address{Princeton University \\ Fine Hall, 304 Washington Rd. \\ Princeton, NJ 08544}
\email{cdefant@princeton.edu}

\begin{abstract}
We introduce ``fertility Wilf equivalence," ``strong fertility Wilf equivalence," and ``postorder Wilf equivalence," three variants of Wilf equivalence for permutation classes that formalize some phenomena that have appeared in the study of West's stack-sorting map. We introduce ``sliding operators" and show that they induce useful bijections among sets of valid hook configurations. Combining these maps with natural decompositions of valid hook configurations, we give infinitely many examples of fertility, strong fertility, and postorder Wilf equivalences. As a consequence, we obtain infinitely many joint equidistribution results concerning many permutation statistics. In one very special case, we reprove and extensively generalize a result of Bouvel and Guibert. Another case reproves and generalizes a result of the current author. A separate very special case proves and generalizes a conjecture of the current author concerning stack-sorting preimages and the Boolean-Catalan numbers. We end with two open questions.    
\end{abstract}

\maketitle

\bigskip

\section{Introduction}
Throughout this paper, the word ``permutation" refers to a permutation of a set of positive integers. We write permutations in one-line notation. Let $S_n$ denote the set of permutations of the set $[n]:=\{1,\ldots,n\}$. If $\pi$ is a permutation of length $n$, then the \emph{normalization} of $\pi$ is the permutation in $S_n$ obtained by replacing the $i^\text{th}$-smallest entry in $\pi$ with $i$ for all $i$. We say a permutation is \emph{normalized} if it is equal to its normalization (equivalently, if it is an element of $S_n$ for some $n$). The set $S_0$ contains one element: the empty permutation. 

\begin{definition}\label{Def2}
Given $\tau\in S_m$, we say a permutation $\sigma=\sigma_1\cdots\sigma_n$ \emph{contains the pattern} $\tau$ if there exist indices $i_1<\cdots<i_m$ in $[n]$ such that the normalization of $\sigma_{i_1}\cdots\sigma_{i_m}$ is $\tau$. We say $\sigma$ \emph{avoids} $\tau$ if it does not contain $\tau$. Let $\Av(\tau^{(1)},\tau^{(2)},\ldots)$ denote the set of normalized permutations that avoid the patterns $\tau^{(1)},\tau^{(2)},\ldots$ (this sequence of patterns could be finite or infinite). A \emph{permutation class} is a set of permutations that is of the form $\Av(\tau^{(1)},\tau^{(2)},\ldots)$ for some patterns $\tau^{(1)},\tau^{(2)},\ldots$. Let $\Av_n(\tau^{(1)},\tau^{(2)},\ldots)=\Av(\tau^{(1)},\tau^{(2)},\ldots)\cap S_n$. 
\end{definition}

One of the central definitions in the study of permutation patterns is that of Wilf equivalence. We say two permutation classes $\Av(\tau^{(1)},\tau^{(2)},\ldots)$ and $\Av(\tau'^{(1)},\tau'^{(2)},\ldots)$ are \emph{Wilf equivalent} if $|\Av_n(\tau^{(1)},\tau^{(2)},\ldots)|=|\Av_n(\tau'^{(1)},\tau'^{(2)},\ldots)|$ for all $n\geq 0$. For example, it is well known that $\Av(\tau)$ and $\Av(\tau')$ are Wilf equivalent whenever $\tau,\tau'\in S_3$. There are many examples of ``trivial" Wilf equivalences that arise from basic symmetries, but there are also several interesting examples of nontrivial Wilf equivalences \cite{Bloom, Burstein, Linton, Stankova}. 

It is difficult to overstate the importance of permutation patterns in modern combinatorics \cite{Bona, Kitaev}. This area originated in the book \emph{The Art of Computer Programming}  \cite{Knuth}, where Knuth introduced a certain \emph{stack-sorting algorithm} and proved that a permutation is sortable via this algorithm if and only if it avoids the pattern $231$. In his Ph.D. thesis, West \cite{West} modified Knuth's original definition to form a function, which we call the \emph{stack-sorting map} and denote by $s$. The name ``stack-sorting" comes from the original definition of $s$, in which one sends a permutation through a vertical ``stack" according to a certain greedy procedure. A simple alternative definition of $s$ is as follows. First, $s$ maps the empty permutation to itself. If $\pi$ is a permutation of a set of positive integers with largest entry $n$, then we can write $\pi=LnR$. We then simply declare $s(\pi)=s(L)s(R)n$. For example, \[s(43512)=s(43)\,s(12)\,5=s(3)\,4\,s(1)\,2\,5=34125.\] There is now a vast collection of literature concerning the stack-sorting map \cite{Bona, BonaWords, BonaSurvey, BonaSimplicial, BonaSymmetry, Bousquet98, Bousquet, Bouvel, BrandenActions, Branden3, Claesson, Cori, DefantCatalan, DefantCounting, DefantDescents, DefantEnumeration, DefantFertility, DefantPolyurethane, DefantPostorder, DefantPreimages, DefantClass, DefantEngenMiller, DefantKravitz, Dulucq, Dulucq2, Fang, Goulden, Hanna, Ulfarsson, West, Zeilberger}.

West \cite{West} defined the \emph{fertility} of a permutation $\pi$ to be $|s^{-1}(\pi)|$, the number of preimages of $\pi$ under $s$. A priori, computing fertilities of permutations is a difficult task. To support this claim, we note that West went to great lengths to find formulas for the fertilities of the very specific permutations of the forms \[23\cdots k1(k+1)\cdots n,\quad 12\cdots(k-2)k(k-1)(k+1)\cdots n,\quad\text{and}\quad k12\cdots(k-1)(k+1)\cdots n.\] In \cite{Bousquet}, Bousquet-M\'elou defined a permutation to be \emph{sorted} if its fertility is positive. She gave an algorithm for determining whether or not a given permutation is sorted and stated that it would be interesting to find a general method to compute the fertility of any given permutation. This was accomplished in much greater generality in \cite{DefantPostorder} and \cite{DefantClass} using new combinatorial objects called ``valid hook configurations." We review the theory of valid hook configurations and their applications to computing fertilities in Section \ref{Sec:VHCs}. 

Most of the questions that researchers have asked about the stack-sorting map can be phrased in terms of preimages of sets of permutations under $s$ \cite{BonaWords, BonaSurvey, BonaSymmetry, Bousquet98, Bousquet, Bouvel, Claesson, DefantCatalan, DefantCounting, DefantDescents, DefantEnumeration, DefantFertility, DefantPolyurethane, DefantPostorder, DefantPreimages, DefantClass, DefantEngenMiller, DefantKravitz, Knuth, Hanna, Ulfarsson, West, Zeilberger}. For example, Bouvel and Guibert \cite{Bouvel} studied permutations that could be sorted via two iterations of the stack-sorting map and a given dihedral symmetry; their results can be reinterpreted as formulas for the sizes of $s^{-1}(\Av_n(132))$ and $s^{-1}(\Av_n(312))$. In \cite{DefantCounting, DefantEnumeration, DefantClass}, the current author computed the fertilities of many sets of the form $\Av_n(\tau^{(1)},\ldots,\tau^{(r)})$ for $\tau^{(1)},\ldots,\tau^{(r)}\in S_3$. He also refined these enumerative results according to the statistics $\des$ and $\peak$. Even more classically, the set of $1$-stack-sortable permutations in $S_n$ is $s^{-1}(\Av_n(21))$, while the set of $2$-stack-sortable permutations in $S_n$ is $s^{-1}(\Av_n(231))$ (see \cite{Bona, BonaSurvey, DefantCounting, DefantPreimages} for definitions). One other motivation for studying preimages of permutation classes under $s$ comes from the fact that these sets are often themselves permutation classes. For instance, $s^{-1}(\Av(321))=\Av(34251,35241,45231)$ (see \cite{DefantClass} for more examples). This leads us to define the fertility of a set $A$ of permutations to be $|s^{-1}(A)|$. It turns out that there are many interesting examples of sets of permutations that have the same fertility. For example, Bouvel and Guibert \cite{Bouvel} showed that 
\begin{equation}\label{Eq2}
|s^{-1}(\Av_n(231))|=|s^{-1}(\Av_n(132))|
\end{equation} for all $n\geq 0$, proving a conjecture of Claesson, Dukes, and Steingrimsson. The current author \cite{DefantClass} also proved that 
\begin{equation}\label{Eq3}
|s^{-1}(\Av_n(132,312))|=|s^{-1}(\Av_n(231,312))|
\end{equation} and conjectured that 
\begin{equation}\label{Eq4} 
|s^{-1}(\Av_n(132,231))|=|s^{-1}(\Av_n(231,312))|.
\end{equation} In fact, we can even trace this phenomenon back to West \cite{West}, who showed that 
\begin{equation}\label{Eq5}
|s^{-1}(\Av_n(132,312,321))|=|s^{-1}(\Av_n(132,231,321))|.
\end{equation} 
This last equation was reproven by Bousquet-M\'elou \cite{Bousquet}. Motivated by these examples, we define a new variant of Wilf equivalence. 

\begin{definition}\label{Def3}
We say that two permutation classes $\Av(\tau^{(1)},\tau^{(2)},\ldots)$ and $\Av(\tau'^{(1)},\tau'^{(2)},\ldots)$ are \emph{fertility Wilf equivalent} if \[|s^{-1}(\Av_n(\tau^{(1)},\tau^{(2)},\ldots))|=|s^{-1}(\Av_n(\tau'^{(1)},\tau'^{(2)},\ldots))|\] for all $n\geq 0$. 
\end{definition}

Previously, \eqref{Eq2}, \eqref{Eq3}, and \eqref{Eq5} were the only known nontrivial examples of fertility Wilf equivalences. In fact, the connections between the pairs of permutation classes in these examples are much deeper than fertility Wilf equivalence. Bouvel and Guibert \cite{Bouvel} listed several permutation statistics and proved that these statistics are jointly equidistributed on $s^{-1}(\Av_n(132))$ and $s^{-1}(\Av_n(231))$ (see Section \ref{Sec:Stats} for definitions). Similarly, the current author \cite{DefantClass} showed that the statistics $\des$ and $\peak$ (defined in Section \ref{Sec:Stats}) are each equidistributed on $s^{-1}(\Av_n(132,312))$ and $s^{-1}(\Av_n(231,312))$ and on $s^{-1}(\Av_n(132,312,321))$ and $s^{-1}(\Av_n(132,231,321))$. It turns out that these statistics are jointly equidistributed on each of these pairs of sets, but we will actually say much more below. 

In Section \ref{Sec:Trees}, we define\footnote{These definitions are not new to this article. In fact, these trees have received a large amount of attention.} a set $\DPT$ of labeled rooted trees called \emph{decreasing plane trees} and consider a special subset of $\DPT$, denoted $\DPT^{(2)}$, whose elements are called \emph{decreasing binary plane trees}. Decreasing binary plane trees can be used to give an alternative definition of the stack-sorting map. Replacing decreasing binary plane trees with other collections of decreasing plane trees yields extensive generalizations of the problem of computing fertilities of permutations. The current author introduced valid hook configurations in \cite{DefantPostorder} in order to develop a method for solving this general problem in a wide variety of natural cases. 

This general point of view involving decreasing plane trees allows us to define a much stronger variant of fertility Wilf equivalence, which we call ``postorder Wilf equivalence." We give this definition and discuss some of its consequences in Section \ref{Sec:Trees}. In particular, postorder Wilf equivalence implies fertility Wilf equivalence. We will also see in Proposition \ref{Prop1} that postorder Wilf equivalence implies a joint equidistribution result concerning a large (uncountable) collection of permutation statistics that we call ``skeletal" statistics, many of which are well-studied. In Section \ref{Sec:VHCs}, we use valid hook configurations to define a separate notion that we call ``strong fertility Wilf equivalence." We will see in Proposition \ref{Prop5} that strong fertility Wilf equivalence implies fertility Wilf equivalence along with some joint equidistribution results for certain permutation statistics. 

In Section \ref{Sec:Main}, we define ``sliding operators" $\swu:\Av(231)\to\Av(132)$ and $\swl:\Av(132)\to\Av(312)$, which are bijections with several useful properties. The map $\swu$ allows us to vastly generalize Bouvel and Guibert's joint equidistribution result concerning stack-sorting preimages of $\Av_n(231)$ and $\Av_n(132)$. In fact, we will give one general unified construction that produces a large infinite collection of pairs of permutation classes that are strongly fertility Wilf equivalent and postorder Wilf equivalent. As a special consequence of one specific case of this construction, we recover the identity \eqref{Eq3}. 

To be completely precise, the joint equidistribution on $s^{-1}(\Av_n(231))$ and $s^{-1}(\Av_n(132))$ of all but one of the statistics that Bouvel and Guibert considered follows from Proposition \ref{Prop1} and Theorem \ref{Thm12}. In order to completely reprove Bouvel and Guibert's full result, we need a short additional argument to handle the last remaining statistic. This is the ``Zeilberger statistic" $\zeil$, which we define in Section \ref{Sec:Stats}.

In a similar vein, we use $\swl$ to prove that $\Av(132,231)$ and $\Av(231,312)$ are strongly fertility Wilf equivalent. In particular, this proves that they are fertility Wilf equivalent, which constitutes the identity \eqref{Eq4} that was conjectured in \cite{DefantClass}. We also show that these sets are not postorder Wilf equivalent. As before, our argument generalizes substantially, and we actually obtain a unified construction that produces a large infinite collection of pairs of permutation classes that are strongly fertility Wilf equivalent. If $\Av(\tau^{(1)},\tau^{(2)},\ldots)$ and $\Av(\tau'^{(1)},\tau'^{(2)},\ldots)$ form one of these pairs, then it follows from Proposition \ref{Prop5} that the statistics $\des$ and $\peak$ are jointly equidistributed on $s^{-1}(\Av_n(\tau^{(1)},\tau^{(2)},\ldots))$ and $s^{-1}(\Av_n(\tau'^{(1)},\tau'^{(2)},\ldots))$ for all $n\geq 0$. At the end of Section \ref{Sec:Main}, we prove the surprising fact that $\zeil$ is also equidistributed on these two sets. In fact, we will prove the stronger statement that $\des$, $\peak$, and $\zeil$ are jointly equidistributed on these sets. In particular, these three statistics are jointly equidistributed on $s^{-1}(\Av_n(132,231))$ and $s^{-1}(\Av_n(231,312))$ for all $n\geq 0$.  

In Section \ref{Sec:Conclusion}, we discuss the implications among the variants of Wilf equivalence introduced throughout the paper. In particular, we show that fertility Wilf equivalence does not imply strong fertility Wilf equivalence. We also end with two open questions.  

\section{Decreasing Plane Trees}\label{Sec:Trees}
A \emph{rooted plane tree} is a rooted tree in which the (possibly empty) subtrees of each vertex are linearly ordered. This is a very broad collection of trees; it is even more broad than the collection of trees discussed in \cite{DefantPostorder}. In fact, there are infinitely many rooted plane trees with just two vertices because such a tree can have arbitrarily many empty subtrees. For example, the root vertex could have $27$ empty subtrees, followed by a child, then followed by $10$ more empty subtrees. Let us stress that the only purpose for considering such a large variety of trees is to demonstrate the versatility of our results. These different types arise in different contexts in combinatorics (binary plane trees, ternary plane trees, Motzkin trees, and many other natural families of trees all fall under the general umbrella of ``rooted plane trees"), so it is nice that our methods can handle all of them uniformly. This also strengthens the definition of postorder Wilf equivalence (hence, strengthening every theorem that yields an example of postorder Wilf equivalence). Figure \ref{Fig1} shows some of the (infinitely many) rooted plane trees with $3$ vertices. A \emph{binary plane tree} is a rooted plane tree in which each vertex has exactly $2$ (possibly empty) subtrees.

\begin{figure}[h]
\begin{center}
\includegraphics[width=.44\linewidth]{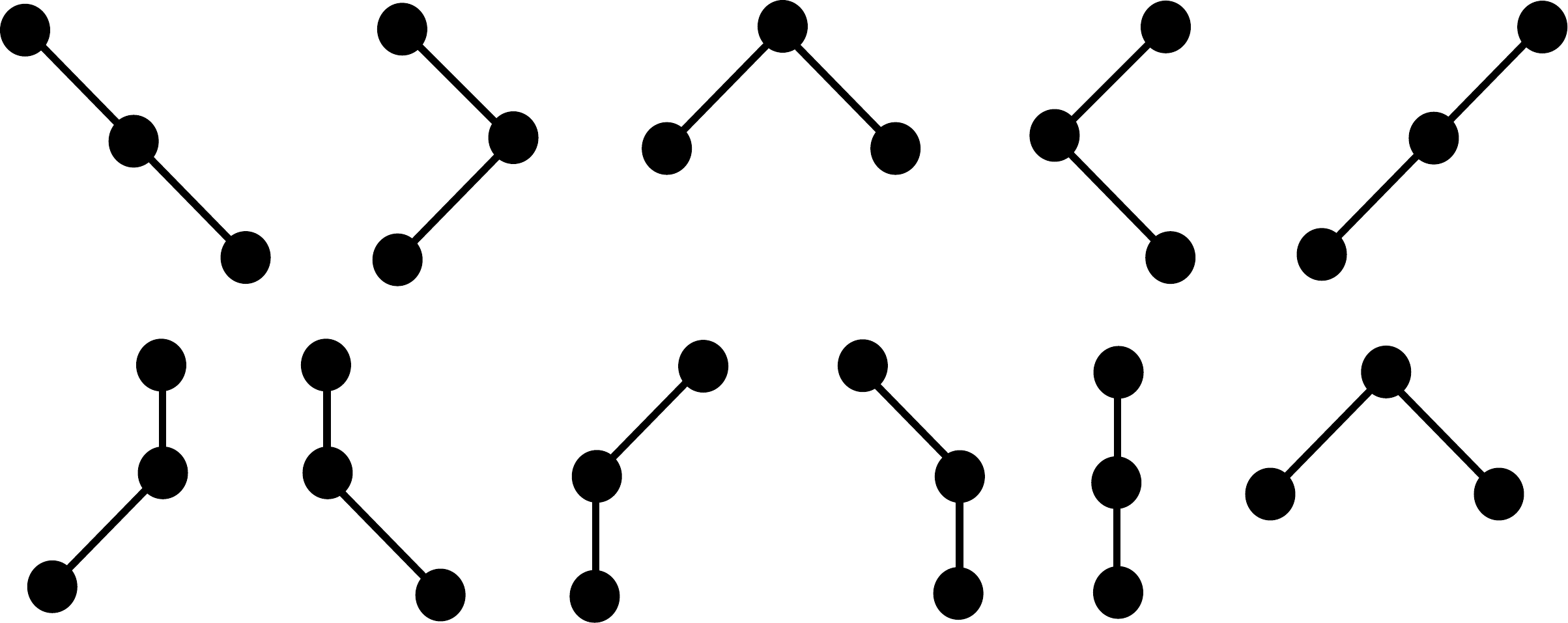}
\caption{Some rooted plane trees on $3$ vertices. The top $5$ are the binary plane trees on $3$ vertices. In the leftmost tree on the bottom, the root vertex has no empty subtrees, while its single child has a left child and an empty right subtree.}
\label{Fig1}
\end{center}  
\end{figure}

If $X$ is a set of positive integers, then a \emph{decreasing plane tree on $X$} is a rooted plane tree in which the vertices are labeled with the elements of $X$ (where each label is used exactly once) such that each nonroot vertex has a label that is strictly smaller than the label of its parent. Figure \ref{Fig2} shows two different decreasing binary plane trees on $\{1,\ldots,7\}$. Let $\DPT$ denote the set of all decreasing plane trees. Let $\DPT^{(2)}\subseteq\DPT$ be the set of decreasing binary plane trees. 

\begin{figure}[t]
\begin{center} 
\includegraphics[height=1.8cm]{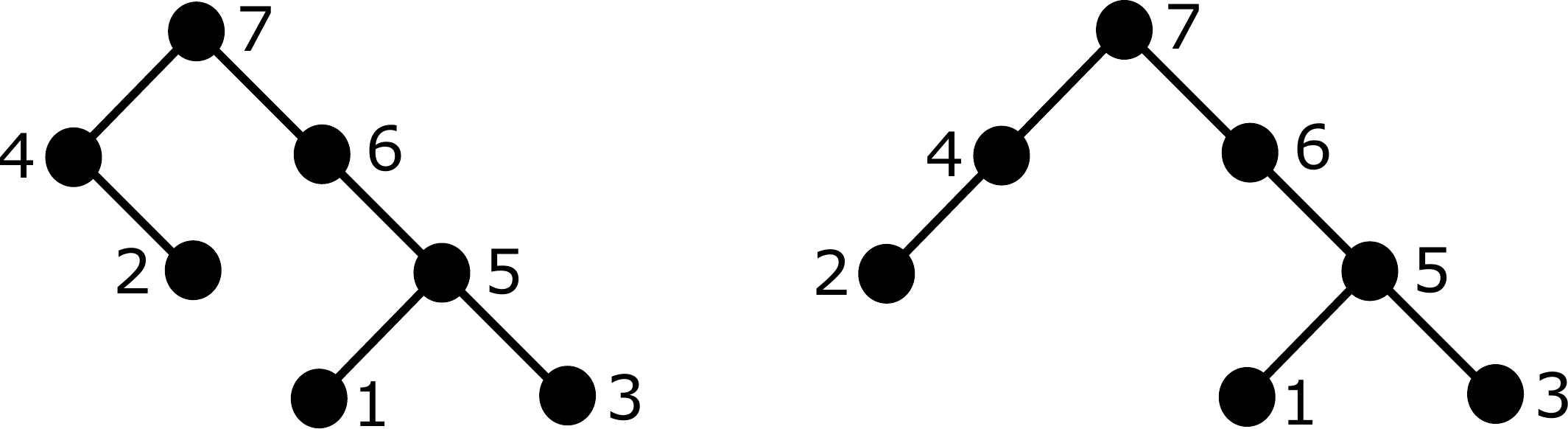}
\end{center}
\caption{Two different decreasing binary plane trees on $\{1,\ldots,7\}$.}\label{Fig2}
\end{figure}

A \emph{tree traversal} is a scheme by which one can read the labels of a labeled tree in some meaningful order to obtain a permutation. One tree traversal, called the \emph{in-order traversal} (sometimes called the \emph{symmetric order traversal}), is only defined on $\DPT^{(2)}$. In order to read a decreasing binary plane tree in in-order, we read the left subtree of the root in in-order, then read the label of the root, and finally read the right subtree of the root in in-order. Let $I(T)$ denote the in-order reading of a decreasing binary plane tree $T$. The map $I$ is a bijection from the set of decreasing binary plane trees on a set $X$ to the set of permutations of $X$ \cite{Bona, Stanley}. Under this bijection, the trees on the left and right in Figure \ref{Fig2} correspond to the permutations $4276153$ and $2476153$, respectively. 

Another tree traversal, called the \emph{postorder traversal}, is defined on all decreasing plane trees. We read a decreasing plane tree in postorder by reading the subtrees of the root from left to right (each in postorder) and then reading the label of the root. Both trees in Figure \ref{Fig2} have postorder $2413567$. Letting $P(T)$ be the postorder reading of a decreasing plane tree $T$, we find that $P$ is a map from $\DPT$ to the set of all permutations. The basic yet fundamental connection between the stack-sorting map and decreasing plane trees comes from the identity \cite{Bona}
\begin{equation}\label{Eq6}
s=P\circ I^{-1}.
\end{equation} It follows from this identity that the fertility of a permutation is equal to the number of decreasing binary plane trees whose postorders are that permutation. In symbols, this says that 
\begin{equation}\label{Eq7}
|s^{-1}(\pi)|=\left|P^{-1}(\pi)\cap \DPT^{(2)}\right|.
\end{equation} 
Therefore, we can vastly generalize the problem of computing the fertility of a permutation $\pi$ to the problem of computing \[\left|P^{-1}(\pi)\cap Y\right|,\] where $Y$ is an arbitrary subset of $\DPT$. In \cite{DefantPostorder}, the current author developed a method for solving this problem for a wide variety of sets $Y$. 

The \emph{skeleton} of a decreasing plane tree $T$ is the rooted plane tree obtained by removing the labels from $T$. If $\mathscr T,\mathscr T'\subseteq\DPT$, then we say a map $\psi:\mathscr T\to\mathscr T'$ is \emph{skeleton-preserving} if $T$ and $\psi(T)$ have the same skeleton for all $T\in\mathscr T$. We end this section with one of the main definitions of this paper.   

\begin{definition}\label{Def4}
We say the permutation classes $\Av(\tau^{(1)},\tau^{(2)},\ldots)$ and $\Av(\tau'^{(1)},\tau'^{(2)},\ldots)$ are \emph{postorder Wilf equivalent} if there exists a skeleton-preserving bijection \[\eta:P^{-1}(\Av(\tau^{(1)},\tau^{(2)},\ldots))\to P^{-1}(\Av(\tau'^{(1)},\tau'^{(2)},\ldots)).\] 
\end{definition}

\section{Permutation Statistics}\label{Sec:Stats}
A \emph{permutation statistic} is a function $f$ from the set of all permutations to $\mathbb N\cup\{0\}$ such that $f(\pi)=f(\pi')$ whenever $\pi$ and $\pi'$ have the same normalization. Note that a permutation statistic is completely determined by its values on normalized permutations. We now set the stage for subsequent sections with notation and terminology concerning permutation statistics. We also prove a proposition that elucidates the strength of postorder Wilf equivalence. 

A \emph{descent} of a permutation $\pi=\pi_1\cdots\pi_n$ is an index $i\in[n-1]$ such that $\pi_i>\pi_{i+1}$. An \emph{ascent} of $\pi$ is an index $i\in[n-1]$ such that $\pi_i<\pi_{i+1}$. A \emph{peak} of $\pi$ is an index $i\in\{2,\ldots,n-1\}$ such that $\pi_{i-1}<\pi_i>\pi_{i+1}$. The \emph{descent set} of $\pi$, denoted $\Des(\pi)$, is the set of descents of $\pi$. One of the most important permutation statistics is $\des$, which is defined by $\des(\pi)=|\Des(\pi)|$. Let $\peak(\pi)$ denote the number of peaks of $\pi$. 

A \emph{left-to-right maximum} of $\pi=\pi_1\cdots\pi_n$ is an entry $\pi_i$ such that $\pi_j<\pi_i$ whenever $1\leq j\leq i-1$. A \emph{right-to-left maximum} of $\pi=\pi_1\cdots\pi_n$ is an entry $\pi_i$ such that $\pi_j<\pi_i$ whenever $i+1\leq j\leq n$. Let $\text{lmax}(\pi)$ and $\rmax(\pi)$ denote the number of left-to-right maxima of $\pi$ and the number of right-to-left maxima of $\pi$, respectively.  

The \emph{Zeilberger statistic}, which originated in Zeilberger's study of $2$-stack-sortable permutations \cite{Zeilberger} and has received attention in subsequent articles such as \cite{Bousquet98, Bouvel, Claesson2, DefantPolyurethane}, is denoted by $\zeil$. For $\pi\in S_n$, $\zeil(\pi)$ is defined to be the largest integer $m$ such that the entries $n,n-1,\ldots,n-m+1$ appear in decreasing order in $\pi$.  
 
The \emph{tail length} of a permutation $\pi=\pi_1\cdots\pi_n\in S_n$, denoted $\tl(\pi)$, is the smallest
nonnegative integer $\ell$ such that $\pi_{n-\ell}\neq n-\ell$. We make the convention that $\tl(123\cdots n)=n$. 
For example, \[\tl(35412678)=3,\quad\tl(1324)=1,\quad\text{and}\quad\tl(21453)=0.\] The tail length is a new statistic that was introduced in \cite{DefantEngenMiller}; it is useful in the study of the stack-sorting map \cite{DefantEngenMiller, DefantCounting, DefantDescents, DefantEnumeration} and will play a crucial role for us in Section \ref{Sec:Main}. In what follows, recall the in-order reading $I$ and the notion of the skeleton of a decreasing plane tree from Section \ref{Sec:Trees}.

\begin{definition}\label{Def11}
The \emph{skeleton} of a permutation $\pi$ is the skeleton of $I^{-1}(\pi)$. We say a permutation statistic $f$ is \emph{skeletal} if for every permutation $\pi$, $f(\pi)$ only depends on the skeleton of $\pi$. 
\end{definition}

Let $\text{LenDes}$ denote the set of permutation statistics $f$ such that $f(\pi)$ only depends on the length of $\pi$ and the descent set $\Des(\pi)$. The set of statistics discussed in \cite{Bouvel} (see \cite{Bouvel,Claesson2} for all of their definitions) is 
\begin{equation}\label{Eq8}
\text{LenDes}\cup\{\text{lmax},\rmax,\zeil,\text{indmax},\text{slmax},\text{slmax}\circ\text{rev}\}.
\end{equation} 
All of the statistics in \eqref{Eq8} except $\zeil$ are skeletal. For example, one can show that $i$ is a descent of $\pi$ if and only if the vertex whose label is read $i^\text{th}$ in the in-order traversal of $I^{-1}(\pi)$ has a right child. Therefore, all of the statistics in $\text{LenDes}$ are skeletal. Even though $\tl,\rmax,\text{lmax},\text{indmax},\text{slmax}$, and $\text{slmax}\circ\text{rev}$ are not in $\text{LenDes}$, they are still skeletal. 

The following lemma connects the statistics $\zeil,\rmax$, and $\tl$ with the stack-sorting map $s$. We will use it to understand the statistic $\zeil$ in the proof of Corollary \ref{Cor3}. 

\begin{lemma}\label{Lem4}
For every permutation $\sigma$, we have \[\zeil(\sigma)=\min\{\rmax(\sigma),\tl(s(\sigma))\}.\]
\end{lemma}
\begin{proof}
Without loss of generality, we may assume $\sigma$ is normalized. Choose $\sigma\in S_n$, and put $c=\zeil(\sigma)$. We can write
\[\sigma=\mu^{(0)}\,n\,\mu^{(1)}(n-1)\,\mu^{(2)}\cdots\mu^{(c-1)}(n-c+1)\,\mu^{(c)},\] where $n-c$ does not appear in the subpermutation $\mu^{(c)}$. Since $n,n-1,\ldots,n-c+1$ are right-to-left maxima of $\sigma$, we have $c\leq\rmax(\sigma)$. By the definition of the stack-sorting map, we have \[s(\sigma)=s(\mu^{(0)})s(\mu^{(1)})\cdots s(\mu^{(c-1)})s(\mu^{(c)})(n-c+1)\cdots(n-1)n.\] This shows that $c\leq \tl(s(\sigma))$. We now know that $c\leq\min\{\rmax(\sigma),\tl(s(\sigma))\}$, so it suffices to prove the reverse inequality. If $\mu^{(c)}$ is empty, then $c=\rmax(\sigma)\geq\min\{\rmax(\sigma),\tl(s(\sigma))\}$. Therefore, we may assume $\mu^{(c)}$ is nonempty. The entry in $s(\sigma)$ immediately preceding $n-c+1$ is an entry in $s(\mu^{(c)})$. This is also an entry in $\mu^{(c)}$, so it is not $n-c$. Hence, $c=\tl(s(\sigma))\geq\min\{\rmax(\sigma),\tl(s(\sigma))\}$. 
\end{proof} 

We end this section by discussing joint equidistribution of permutation statistics and how it relates to postorder Wilf equivalence.
 
\begin{definition}\label{Def6} 
Let $A$ and $A'$ be sets of normalized permutations. Let $\mathcal E$ be a set of permutation statistics. We say the elements of $\mathcal E$ are \emph{jointly equidistributed on} $A$ \emph{and} $A'$ if there is a bijection $g:A\to A'$ such that $f(g(\pi))=f(\pi)$ for all $\pi\in A$ and all $f\in\mathcal E$. 
\end{definition} 

We stated in the introduction that Bouvel and Guibert \cite{Bouvel} proved (using different language) that $\Av(231)$ and $\Av(132)$ are fertility Wilf equivalent. In fact, they proved the much stronger statement that the statistics listed in \eqref{Eq8} are jointly equidistributed on $s^{-1}(\Av_n(231))$ and $s^{-1}(\Av_n(132))$ for all $n\geq 0$. The following proposition tells us that the joint equidistribution of all of these statistics other than $\zeil$ on $s^{-1}(\Av_n(231))$ and $s^{-1}(\Av_n(132))$ is a special consequence of the fact that $\Av(231)$ and $\Av(132)$ are postorder Wilf equivalent. This, in turn, is a special case of Theorem \ref{Thm12} in Section \ref{Sec:Main}. We will also be able to add $\zeil$ to this list of equidistributed statistics in Corollary~\ref{Cor3} with the help of Lemma \ref{Lem4}.  

\begin{proposition}\label{Prop1}
Let $\Av(\tau^{(1)},\tau^{(2)},\ldots)$ and $\Av(\tau'^{(1)},\tau'^{(2)},\ldots)$ be permutation classes that are postorder Wilf equivalent. For every $n\geq 0$, all skeletal statistics are jointly equidistributed on $s^{-1}(\Av_n(\tau^{(1)},\tau^{(2)},\ldots))$ and $s^{-1}(\Av_n(\tau'^{(1)},\tau'^{(2)},\ldots))$. In particular, these two permutation classes are fertility Wilf equivalent. 
\end{proposition}
\begin{proof}
According to Definition \ref{Def4}, there exists a skeleton-preserving bijection \[\eta:P^{-1}(\Av(\tau^{(1)},\tau^{(2)},\ldots))\to P^{-1}(\Av(\tau'^{(1)},\tau'^{(2)},\ldots)).\] For each positive integer $n$, the map $\eta$ sends decreasing binary plane trees on $[n]$ to decreasing binary plane trees on $[n]$. In other words, it induces a skeleton-preserving bijection \begin{equation}\label{Eq18}
\widetilde\eta:P^{-1}(\Av_n(\tau^{(1)},\tau^{(2)},\ldots))\cap\DPT^{(2)}\to P^{-1}(\Av_n(\tau'^{(1)},\tau'^{(2)},\ldots))\cap\DPT^{(2)}.
\end{equation} Using the identity $s=P\circ I^{-1}$ from \eqref{Eq6}, we find that \[I(P^{-1}(\Av_n(\tau^{(1)},\tau^{(2)},\ldots))\cap\DPT^{(2)})=s^{-1}(\Av_n(\tau^{(1)},\tau^{(2)},\ldots)).\] It follows that the map $\eta^{\#}:=I\circ\widetilde\eta\circ I^{-1}$ is a bijection from $s^{-1}(\Av_n(\tau^{(1)},\tau^{(2)},\ldots))$ to $s^{-1}(\Av_n(\tau'^{(1)},\tau'^{(2)},\ldots))$. Because $\widetilde\eta$ is skeleton-preserving, $I^{-1}(\sigma)$ and $I^{-1}(\eta^{\#}(\sigma))$ have the same skeleton. This means that $\sigma$ and $\eta^{\#}(\sigma)$ have the same skeleton, so $f(\sigma)=f(\eta^{\#}(\sigma))$ for every $\sigma\in s^{-1}(\Av_n(\tau^{(1)},\tau^{(2)},\ldots))$ and every skeletal statistic $f$.   
\end{proof}

\begin{remark}\label{Rem4}
In the proof of Proposition \ref{Prop1}, we only used the hypothesis that the permutation classes were postorder Wilf equivalent in order to deduce the existence of the skeleton-preserving bijection $\widetilde\eta$ in \eqref{Eq18}. We never used the full strength of the hypothesis that there is a skeleton-preserving bijection between the \emph{much larger} sets $P^{-1}(\Av_n(\tau^{(1)},\tau^{(2)},\ldots))$ and $P^{-1}(\Av_n(\tau'^{(1)},\tau'^{(2)},\ldots))$. In other words, we really only used the fact that the permutation classes are postorder Wilf equivalent when we restrict our attention to decreasing binary plane trees. Therefore, stating that two permutation classes are postorder Wilf equivalent is much stronger than stating that they satisfy the conclusion of Proposition \ref{Prop1}. To phrase this more precisely, let us say that two permutation classes $\Av(\tau^{(1)},\tau^{(2)},\ldots)$ and $\Av(\tau'^{(1)},\tau'^{(2)},\ldots)$ are \emph{binary postorder Wilf equivalent} if there exists a skeleton-preserving bijection $\widetilde\eta$ as in \eqref{Eq18}. In Section \ref{Sec:Conclusion}, we show that $\Av(123)$ and $\Av(123,3214)$ are binary postorder Wilf equivalent but not postorder Wilf equivalent. \hspace*{\fill}$\lozenge$ 
\end{remark}

\section{Valid Hook Configurations}\label{Sec:VHCs} 

In \cite{DefantPostorder}, the current author introduced ``valid hook configurations" in order to solve the problem of computing $\left|P^{-1}(\pi)\cap Y\right|$, where $P$ denotes the postorder traversal defined in Section \ref{Sec:Main} and $Y$ is an arbitrary set of decreasing plane trees. We wish to break with the notational conventions introduced in that article. We use the term ``valid hook configuration" to refer to a slight variant of a specific type of object considered in \cite{DefantPostorder}. The objects turn out to have interesting combinatorial properties in their own right \cite{DefantCatalan, DefantMotzkin, DefantEngenMiller, Hanna, Maya}. In this section, we state the main fertility formulas that connect valid hook configurations with the stack-sorting map. We also define strong fertility Wilf equivalence and discuss some of its consequences. 

The first part of a valid hook configuration is a permutation $\pi=\pi_1\cdots\pi_n$. We use the example permutation \[2\,\,7\,\,3\,\,5\,\,9\,\,10\,\,11\,\,4\,\,8\,\,1\,\,6\,\,12\,\,13\,\,14\,\,15\,\,16\] throughout this section. The \emph{plot} of $\pi$ is obtained by plotting the points $(i,\pi_i)$ for all $i\in[n]$. If $i\in[n-1]$ is a descent of $\pi$ (recall that this means $\pi_i>\pi_{i+1}$), then we call the point $(i,\pi_i)$ a \emph{descent top of the plot of} $\pi$. The left image in Figure \ref{Fig8} shows the plot of our example permutation. 

\begin{figure}[h]
  \centering
  \subfloat[]{\includegraphics[width=.38\linewidth]{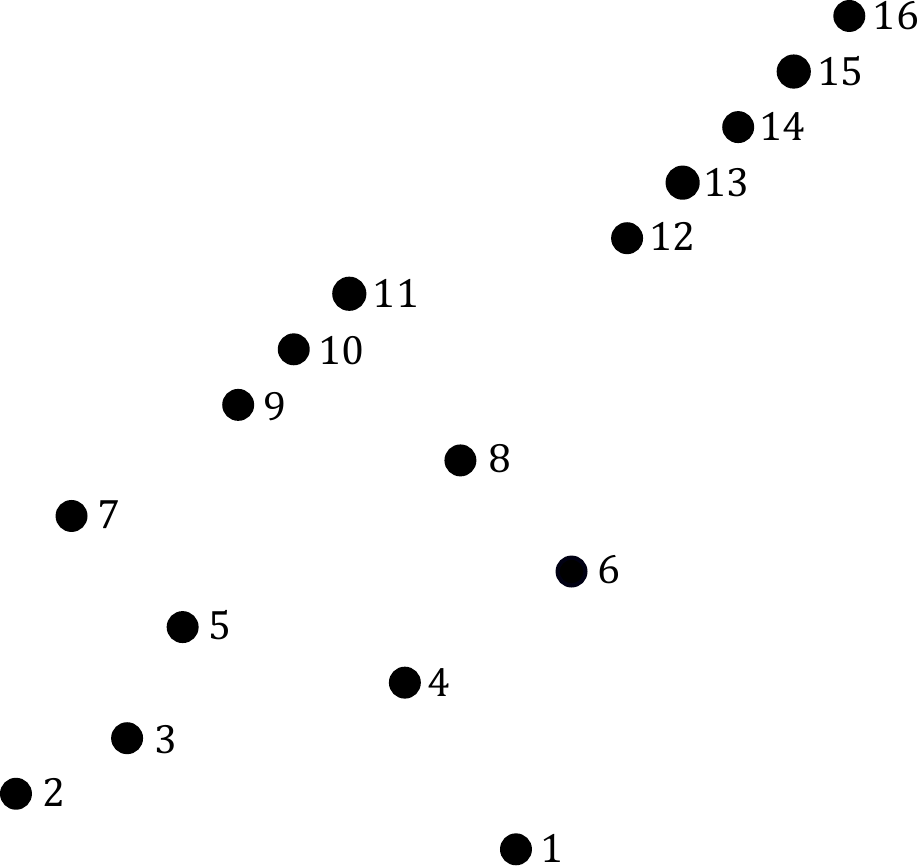}}
  \hspace{1.5cm}
  \subfloat[]{\includegraphics[width=.38\linewidth]{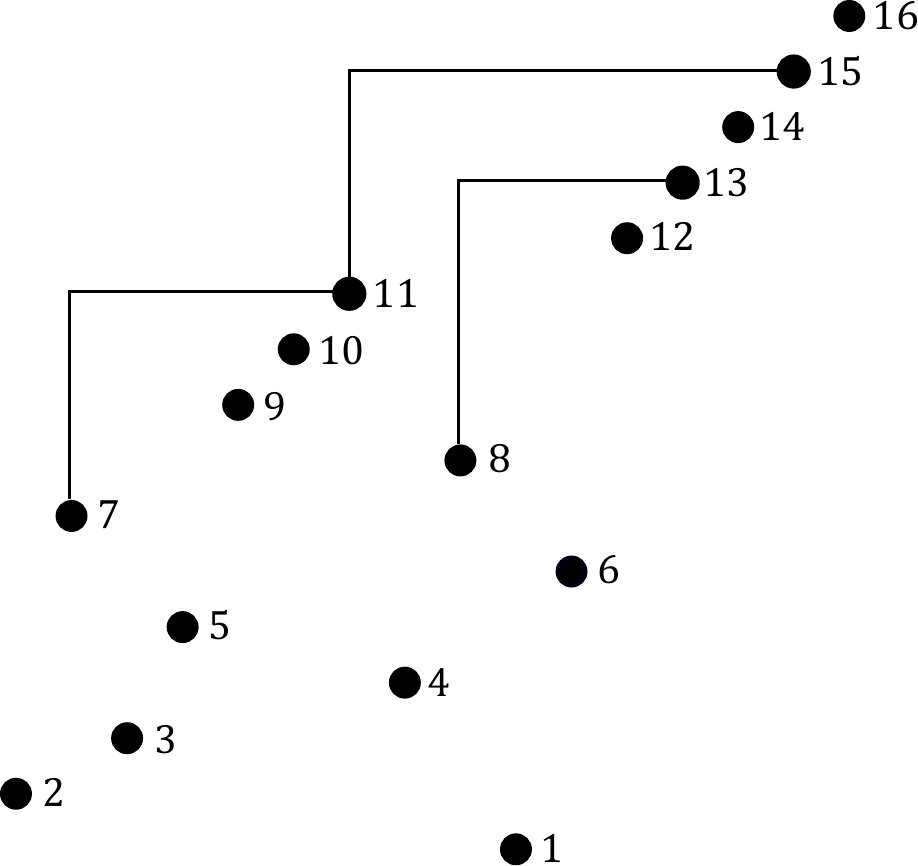}}
  \caption{The left image depicts the plot of the permutation $2\,\,7\,\,3\,\,5\,\,9\,\,10\,\,11\,\,4\,\,8\,\,1\,\,6\,\,12\,\,13\,\,14\,\,15\,\,16$. The right image shows a valid hook configuration of this permutation.}\label{Fig8}
\end{figure}

A \emph{hook} of $\pi$ is drawn by starting at a point $(i,\pi_i)$ in the plot of $\pi$, drawing a line segment vertically upward, and then drawing a line segment horizontally to the right until reaching another point $(j,\pi_j)$. This only makes sense if $i<j$ and $\pi_i<\pi_j$. The point $(i,\pi_i)$ is called the \emph{southwest endpoint} of the hook, while $(j,\pi_j)$ is called the \emph{northeast endpoint}. The right image in Figure \ref{Fig8} shows the plot of our example permutation with three hooks. The southwest endpoints of the hooks are $(2,7)$, $(7,11)$, $(9,8)$, and the corresponding northeast endpoints are $(7,11)$, $(15,15)$, $(13,13)$. 

\begin{figure}[t]
\begin{center}
\includegraphics[width=.66\linewidth]{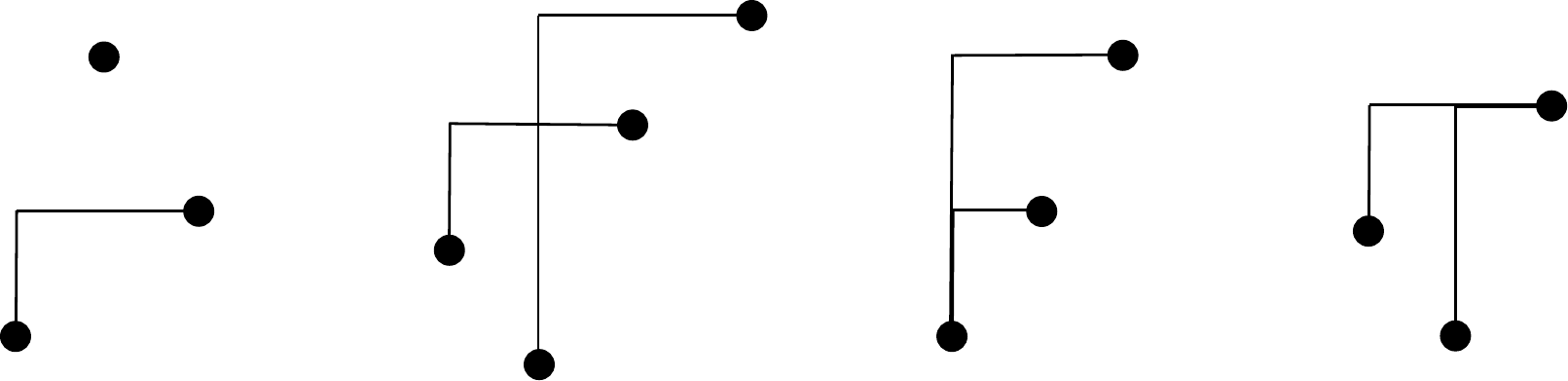}
\caption{Four placements of hooks that are forbidden in a valid hook configuration.}
\label{Fig9}
\end{center}  
\end{figure}

\begin{definition}\label{Def9}
Let $\pi$ be a permutation of length $n$ with $k$ descents. A \emph{valid hook configuration} of $\pi$ is a tuple $(H_1,\ldots,H_k)$ of hooks of $\pi$ subject to the following constraints: 

\begin{enumerate}[1.]
\item The southwest endpoints of the hooks are precisely the descent tops of the plot of $\pi$. 

\item A point in the plot of $\pi$ cannot lie directly above a hook. 

\item Hooks cannot intersect or overlap each other except in the case that the northeast endpoint of one hook is the southwest endpoint of the other. 
\end{enumerate}  
Let $\VHC(\pi)$ denote the set of valid hook configurations of $\pi$. We make the convention that a valid hook configuration includes its underlying permutation as part of its definition. In other words, $\VHC(\pi)$ and $\VHC(\pi')$ are disjoint whenever $\pi$ and $\pi'$ are distinct.
\end{definition}

A valid hook configuration of $\pi$ induces a coloring of the plot of $\pi$. To color the plot, draw a ``sky" over the entire diagram and assign a color to the sky. Assign arbitrary distinct colors other than the one used to color the sky to the $k$ hooks in the valid hook configuration. There are $k$ northeast endpoints of hooks, and these points remain uncolored. However, all of the other $n-k$ points will be colored. In order to decide how to color a point $(i,\pi_i)$ that is not a northeast endpoint, imagine that this point looks directly upward. If this point sees a hook when looking upward, it receives the same color as the hook that it sees. If the point does not see a hook, it must see the sky, so it receives the color of the sky. However, if $(i,\pi_i)$ is the southwest endpoint of a hook, then it must look around (on the left side of) the vertical part of that hook (see Figure \ref{Fig11}). 

To summarize, we started with a permutation $\pi$ with $k$ descents. We chose a valid hook configuration $(H_1,\ldots,H_k)$ of $\pi$ by drawing $k$ hooks according to Conditions 1, 2, and 3 in Definition \ref{Def9}. This valid hook configuration then induced a coloring of the plot of $\pi$. Specifically, $n-k$ points were colored, and $k+1$ colors were used (one for each hook and one for the sky). Let $q_i$ be the number of points given the same color as $H_i$, and let $q_0$ be the number of points given the same color as the sky. Then $(q_0,\ldots,q_k)$ is a composition\footnote{Throughout this paper, a \emph{composition of }$b$ \emph{into} $a$ \emph{parts} is an $a$-tuple of \emph{positive} integers whose sum is $b$.} of $n-k$ into $k+1$ parts; we say the valid hook configuration \emph{induces} this composition. Let $\mathcal V(\pi)$ be the set of compositions induced by valid hook configurations of $\pi$. We call the elements of $\mathcal V(\pi)$ the \emph{valid compositions} of $\pi$.  

\begin{figure}[t]
\begin{center}
\includegraphics[width=.42\linewidth]{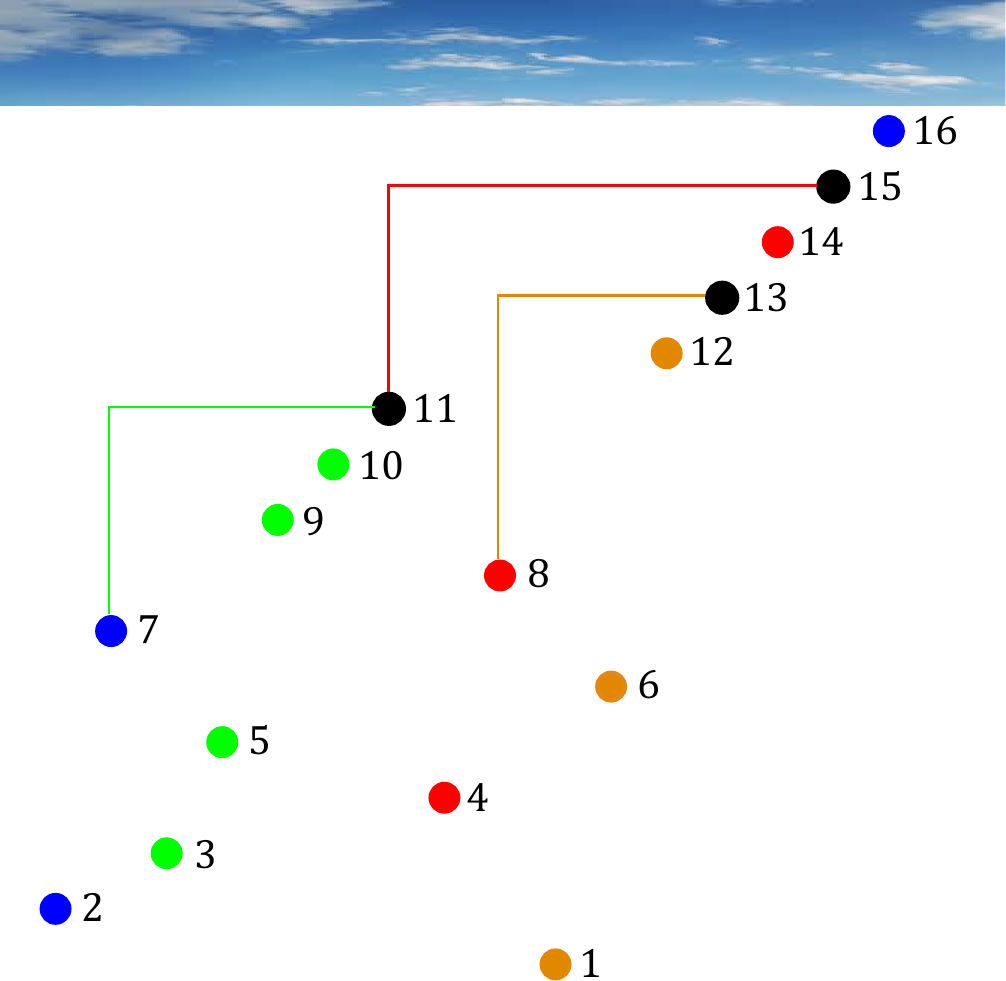}
\caption{The coloring induced by a valid hook configuration.}
\label{Fig11}
\end{center}  
\end{figure}

We frequently make tacit use of the following result, which is Lemma 3.1 in \cite{DefantPostorder}. 

\begin{theorem}[\!\!\cite{DefantPostorder}]\label{Thm9}
Each valid composition of a permutation $\pi$ is induced by a unique valid hook configuration of $\pi$. 
\end{theorem}

The next theorem, which follows from the results in Section 5 of \cite{DefantPostorder}, has proven useful in \cite{DefantCounting,DefantPreimages,DefantClass}. 
Let $L(r,i,j)$ be the number of binary plane trees with $r$ vertices, $i-1$ right edges, and $j$ leaves. Let $L_r(x,y)=\sum_{i=1}^r\sum_{j=1}^rL(r,i,j)x^iy^j$. Let $C_r=L_r(1,1)=\frac{1}{r+1}{2r\choose r}$ be the $r^\text{th}$ Catalan number.

\begin{theorem}[\!\!\cite{DefantPostorder}]\hspace{-.15cm}\footnote{The article \cite{DefantPostorder} gives a general construction that allows one to produce decreasing plane trees of various types that have a specified permutation as their postorder readings. This leads to numerous analogues and generalizations of Theorem \ref{Thm10}. For example, a very special consequence of Theorem 4.1 in that article is that the number of decreasing Motzkin trees with postorder $\pi$ is $\displaystyle\sum_{(q_0,\ldots,q_k)\in\mathcal V(\pi)}\prod_{t=0}^kM_{q_t-1}$, where $M_r$ is the $r^\text{th}$ Motzkin number.}\label{Thm10}
Let $\pi=\pi_1\cdots\pi_n$ be a permutation with $\des(\pi)=k$. We have \[\sum_{\sigma\in s^{-1}(\pi)}x^{\des(\sigma)+1}y^{\peak(\sigma)+1}=\sum_{(q_0,\ldots,q_k)\in\mathcal V(\pi)}\prod_{t=0}^kL_{q_t}(x,y).\] In particular, \[|s^{-1}(\pi)|=\sum_{(q_0,\ldots,q_k)\in\mathcal V(\pi)}\prod_{t=0}^kC_{q_t}.\]
\end{theorem}

We end this section by defining and discussing strong fertility Wilf equivalence. Let us first fix some simple terminology and notation. A \emph{partition} is a composition whose parts appear in nonincreasing order. The \emph{type} of a composition $c$ is the partition obtained by rearranging the parts of $c$ into nonincreasing order. For instance, the type of $(1,3,4,1)$ is $(4,3,1,1)$. Define the \emph{type} of a valid hook configuration $\mathcal H$ to be the type of the valid composition induced by $\mathcal H$. For example, the valid hook configuration in Figure \ref{Fig11} induces the valid composition $(3,4,3,3)$, so it has type $(4,3,3,3)$. If $B$ and $B'$ are sets of valid hook configurations, we say a function $\theta:B\to B'$ is \emph{type-preserving} if $\mathcal H$ and $\theta(\mathcal H)$ have the same type for every $\mathcal H\in B$. Given a set $A$ of permutations, let $\VHC(A)=\bigcup_{\pi\in A}\VHC(\pi)$ be the set of valid hook configurations of the elements of $A$. 

\begin{definition}\label{Def10}
We say the permutation classes $\Av(\tau^{(1)},\tau^{(2)},\ldots)$ and $\Av(\tau'^{(1)},\tau'^{(2)},\ldots)$ are \emph{strongly fertility Wilf equivalent} if there exists a type-preserving bijection \[\theta:\VHC(\Av(\tau^{(1)},\tau^{(2)},\ldots))\to\VHC(\Av(\tau'^{(1)},\tau'^{(2)},\ldots)).\]
\end{definition}

\begin{remark}\label{Rem6}
If $\theta$ is as in Definition \ref{Def10}, then \[\theta(\VHC(\Av_n(\tau^{(1)},\tau^{(2)},\ldots)))=\VHC(\Av_n(\tau'^{(1)},\tau'^{(2)},\ldots))\] for all $n\geq 0$. Indeed, suppose $\mathcal H\in\VHC(\Av(\tau^{(1)},\tau^{(2)},\ldots))$. Let $\pi$ and $\pi'$ be the underlying permutations of $\mathcal H$ and $\theta(\mathcal H)$, respectively (these are uniquely determined according to the last part of Definition \ref{Def9}). Suppose $\pi\in S_n$ and $\pi'\in S_{n'}$. Let $k=\des(\pi)$ and $k'=\des(\pi')$. By our discussion above, the valid composition induced by $\mathcal H$ is a composition of $n-k$ into $k+1$ parts. Similarly, the valid composition induced by $\theta(\mathcal H)$ is a composition of $n'-k'$ into $k'+1$ parts. Because $\theta$ is type-preserving, we must have $n-k=n'-k'$ and $k+1=k'+1$. Hence, $n=n'$. \hspace*{\fill}$\lozenge$ 
\end{remark}

\begin{proposition}\label{Prop5}
Let $\Av(\tau^{(1)},\tau^{(2)},\ldots)$ and $\Av(\tau'^{(1)},\tau'^{(2)},\ldots)$ be permutation classes that are strongly fertility Wilf equivalent. For every $n\geq 0$, the statistics $\des$ and $\peak$ are jointly equidistributed on $s^{-1}(\Av_n(\tau^{(1)},\tau^{(2)},\ldots))$ and $s^{-1}(\Av_n(\tau'^{(1)},\tau'^{(2)},\ldots))$. In particular, $\Av(\tau^{(1)},\tau^{(2)},\ldots,)$ and $\Av(\tau'^{(1)},\tau'^{(2)},\ldots,)$ are fertility Wilf equivalent. 
\end{proposition}

\begin{proof}
Fix $n\geq 0$. Let \[\theta:\VHC(\Av(\tau^{(1)},\tau^{(2)},\ldots))\to\VHC(\Av(\tau'^{(1)},\tau'^{(2)},\ldots))\] be the type-preserving bijection whose existence is guaranteed by Definition \ref{Def10}. Remark \ref{Rem6} tells us that $\theta(\VHC(\Av_n(\tau^{(1)},\tau^{(2)},\ldots)))=\VHC(\Av_n(\tau'^{(1)},\tau'^{(2)},\ldots))$. Let $L_r(x,y)$ be as in Theorem~\ref{Thm10}. Given a composition $q=(q_0,\ldots,q_k)$, let \[L_q(x,y)=\prod_{t=0}^kL_{q_t}(x,y).\] Let $q^{\mathcal H}$ denote the valid composition induced by a valid hook configuration $\mathcal H$. Because $\theta$ is type-preserving, we have $L_{q^{\mathcal H}}(x,y)=L_{q^{\theta(\mathcal H)}}(x,y)$ for all $\mathcal H\in\VHC(\Av_n(\tau^{(1)},\tau^{(2)},\ldots))$. Invoking Theorems \ref{Thm9} and \ref{Thm10}, we find that 
\begin{align*}
\sum_{\sigma\in s^{-1}(\Av_n(\tau^{(1)},\tau^{(2)},\ldots))}x^{\des(\sigma)+1}y^{\peak(\sigma)+1}&=\sum_{\mathcal H\in\VHC(\Av_n(\tau^{(1)},\tau^{(2)},\ldots))}L_{q^{\mathcal H}}(x,y) \\ 
&=\sum_{\mathcal H\in\VHC(\Av_n(\tau^{(1)},\tau^{(2)},\ldots))}L_{q^{\theta(\mathcal H)}}(x,y) \\ 
&=\sum_{\mathcal H'\in\VHC(\Av_n(\tau'^{(1)},\tau'^{(2)},\ldots))}L_{q^{\mathcal H'}}(x,y) \\ 
&=\sum_{\sigma'\in s^{-1}(\Av_n(\tau'^{(1)},\tau'^{(2)},\ldots))}x^{\des(\sigma')+1}y^{\peak(\sigma')+1}. \qedhere
\end{align*}   
\end{proof}

\section{Main Results}\label{Sec:Main}

In this section, we define the sliding operators $\swu$ and $\swl$. We use $\swu$ to produce infinitely many examples of pairs of postorder Wilf equivalent permutation classes, recovering Bouvel and Guibert's result concerning $s^{-1}(\Av(231))$ and $s^{-1}(\Av(132))$ as a special consequence. We then use $\swl$ to produce infinitely many examples of pairs of strongly fertility Wilf equivalent permutation classes, proving the conjectured identity \eqref{Eq4} as a special consequence. We end this section with a discussion of the Zeilberger statistic and its joint equidistribution with $\des$ and $\peak$ on certain sets.  

For $\pi\in S_n$, let $\rot(\pi)$ (respectively, $\rot^{-1}(\pi)$) be the permutation whose plot is obtained by rotating the plot of $\pi$ counterclockwise (respectively, clockwise) by $90^\circ$. Equivalently, $\rot(\pi)$ is the reverse of the inverse of $\pi$. If $\lambda=\lambda_1\cdots\lambda_\ell\in S_\ell$ and $\mu=\mu_1\ldots\mu_m\in S_m$, then the \emph{direct sum} of $\lambda$ and $\mu$, denoted $\lambda\oplus\mu$, is the permutation in $S_{\ell+m}$ obtained by ``placing $\mu$ above and to the right of $\lambda$." The \emph{skew sum} of $\lambda$ and $\mu$, denoted $\lambda\ominus\mu$, is the permutation in $S_{\ell+m}$ obtained by ``placing $\mu$ below and to the right of $\lambda$." More formally, the $i^\text{th}$ entries of $\lambda\oplus\mu$ and $\lambda\ominus\mu$, respectively, are \[(\lambda\oplus\mu)_i=\begin{cases} \lambda_i & \mbox{if } 1\leq i\leq \ell; \\ \mu_{i-\ell}+\ell & \mbox{if } \ell+1\leq i\leq \ell+m \end{cases}\hspace{.33cm} \text{and}\hspace{.33cm}(\lambda\ominus\mu)_i=\begin{cases} \lambda_i+m & \mbox{if } 1\leq i\leq \ell; \\ \mu_{i-\ell} & \mbox{if } \ell+1\leq i\leq \ell+m. \end{cases}\] Note that $\oplus$ and $\ominus$ are both associative operations on the set of normalized permutations. We say a normalized permutation is \emph{sum indecomposable} if it cannot be written as the direct sum of two shorter permutations. We say a normalized permutation is \emph{skew indecomposable} if it cannot be written as the skew sum of two shorter permutations. 

\begin{definition}\label{Def8}
If $\pi$ is the empty permutation, then $\swu(\pi)=\pi$. If $\pi\in \Av_n(231)$ for some $n\geq 1$, then we can write $\pi=L\oplus (1\ominus R)$ for some normalized permutations $L$ and $R$. We let \[\swu(\pi)=(\swu(L)\oplus 1)\ominus\swu(R).\] For $\pi\in\Av(132)$, let \[\swl(\pi)=\rot^{-1}(\swu(\rot(\pi))).\]
\end{definition}

\begin{remark}\label{Rem7}
The name ``$\swu$" stands for ``southwest up" because $\swu$ has the effect of sliding up points in the southwest region of the plot of $\pi$. Similarly, ``$\swl$" stands for ``southwest left." It is sometimes helpful to extend the definitions of $\swu$ and $\swl$ to permutations of arbitrary sets of positive integers. 
If $\pi$ is a $231$-avoiding permutation of a set $X$ of positive integers and $\pi'$ is the normalization of $\pi$, then we define $\swu(\pi)$ to be the unique permutation of $X$ whose normalization is $\swu(\pi')$. We define $\swl$ on arbitrary $132$-avoiding permutations similarly.  \hspace*{\fill}$\lozenge$ 
\end{remark}

The article \cite{DefantCatalan} discusses and proves several properties of the sliding operators $\swu$ and $\swl$. In the following lemma, we simply state the properties that we will need later. We omit the proofs because they either appear in \cite{DefantCatalan} or are immediate from the definitions we have given. Recall the definitions of $\Des(\pi)$, $\des(\pi)$, and $\tl(\pi)$ from Section \ref{Sec:Stats}. 

\begin{lemma}\label{Lem1}
The maps $\swu$ and $\swl$ have the following properties: 

\begin{itemize}
\item The map $\swu:\Av(231)\to\Av(132)$ is bijective. 
\item The map $\swl:\Av(132)\to\Av(312)$ is bijective. 
\item The restriction of $\swl$ to $\Av(132,231)$ is a bijection from $\Av(132,231)$ to $\Av(231,\!312)$. 
\item We have $\tl(\pi)=\tl(\swu(\pi))$ and $\Des(\pi)=\Des(\swu(\pi))$ for every $\pi\in\Av(231)$. 
\item We have $\tl(\pi)=\tl(\swl(\pi))$ and $\des(\pi)=\des(\swl(\pi))$ for every $\pi\in\Av(132)$.
\end{itemize}
\end{lemma} 

\begin{remark}
Since $\swu$ and $\swl$ are bijections, they have inverses $\swu^{-1}$ and $\swl^{-1}$. These maps are called $\swd$ and $\swr$ in \cite{DefantCatalan}, but we will not use these names here.  \hspace*{\fill}$\lozenge$ 
\end{remark}

Our primary motivation for considering the map $\swu$ comes from the following proposition, which is proven in \cite{DefantPolyurethane} using polyurethane toggles. 

\begin{proposition}\label{Prop2}
For every $\pi\in\Av(231)$, there is a skeleton-preserving bijection \[\eta_\pi:P^{-1}(\pi)\to P^{-1}(\swu(\pi)).\] 
\end{proposition}

The next proposition will provide an important tool for building permutation classes that behave nicely under the map $\swu$. 

\begin{proposition}\label{Prop3}
Let $\tau$ be a permutation such that $\swu(\Av(231,\tau))=\Av(132,\swu(\tau))$. If $\tau$ is sum indecomposable, then \[\swu(\Av(231,\tau\oplus 1))=\Av(132,\swu(\tau\oplus 1)).\] If $\tau$ is skew indecomposable, then \[\swu(\Av(231,1\ominus \tau))=\Av(132,\swu(1\ominus\tau)).\]  
\end{proposition} 

\begin{proof}
We prove the case in which $\tau$ is sum indecomposable; the proof of the case in which $\tau$ is skew indecomposable is similar. We will prove that \[\swu(\Av_n(231,\tau\oplus 1))=\Av_n(132,\swu(\tau\oplus 1))\] for all $n\geq 0$. This is easy if $n\leq 1$, so we may assume $n\geq 2$ and proceed by induction on $n$. Choose $\pi\in\Av_n(231,\tau\oplus 1)$ and $\sigma\in\Av_n(132,\swu(\tau\oplus 1))$. Our goal is to show that $\swu(\pi)\in\Av_n(132,\swu(\tau\oplus 1))$ and $\swu^{-1}(\sigma)\in\Av_n(231,\tau\oplus 1)$. We already know that $\swu(\pi)$ avoids $132$ and $\swu^{-1}(\sigma)$ avoids $231$, so we are left to show that $\swu(\pi)$ avoids $\swu(\tau\oplus 1)$ and $\swu^{-1}(\sigma)$ avoids $\tau\oplus 1$.   

We can write 
\begin{equation}\label{Eq14}
\pi=L\oplus(1\ominus R)\quad \text{and}\quad\sigma=(\widehat L\oplus 1)\ominus\widehat R
\end{equation} so that 
\begin{equation}\label{Eq15}\swu(\pi)=(\swu(L)\oplus 1)\ominus \swu(R)\quad\text{and}\quad\swu^{-1}(\sigma)=\swu^{-1}(\widehat L)\oplus(1\ominus\swu^{-1}(\widehat R)).
\end{equation} 
Because $L,R\in\Av(231,\tau\oplus 1)$, it follows by induction that $\swu(L)$ and $\swu(R)$ avoid $\swu(\tau\oplus 1)$. Similarly, $\swu^{-1}(\widehat L)$ and $\swu^{-1}(\widehat R)$ avoid $\tau\oplus 1$. 

Assume by way of contradiction that $\swu(\pi)$ contains $\swu(\tau\oplus 1)$. Note that $\swu(\tau\oplus 1)=\swu(\tau)\oplus 1$ by the definition of $\swu$. Since $\swu(L)$ and $\swu(R)$ avoid $\swu(\tau\oplus 1)$, it follows from \eqref{Eq15} that $\swu(L)$ contains $\swu(\tau)$. Using our hypothesis, we deduce that \[\swu(L)\not\in\Av(132,\swu(\tau))=\swu(\Av(231,\tau)).\] We know that $L$ avoids $231$ (because $\pi$ does), so $L$ must contain $\tau$. We can now use \eqref{Eq14} to see that $\pi$ contains $\tau\oplus 1$, which is our desired contradiction. 

Next, assume by way of contradiction that $\swu^{-1}(\sigma)$ contains $\tau\oplus 1$. Combining \eqref{Eq15} with the hypothesis that $\tau$ is sum indecomposable, it is straightforward to check that $\swu^{-1}(\widehat L)$ contains $\tau$. This means that \[\widehat L\not\in\swu(\Av(231,\tau))=\Av(132,\swu(\tau)).\] We know that $\widehat L$ avoids $132$ (because $\sigma$ does), so $\widehat L$ must contain $\swu(\tau)$. It is now immediate from \eqref{Eq14} that $\sigma$ contains $\swu(\tau)\oplus 1=\swu(\tau\oplus 1)$, which is a contradiction.    
\end{proof}

Propositions \ref{Prop2} and \ref{Prop3} allow us to produce several examples of postorder Wilf equivalences. The following theorem exhibits infinitely many such examples, but there could certainly be others. Let us first fix some notation. Given $\pi\in S_n$, let 
\begin{equation}\label{Eq22}
\chi_m(\pi)=\begin{cases}  (n+m-1)\cdots(n+3)(n+1)\pi(n+2)(n+4)\cdots(n+m) & \mbox{if } m\equiv 0\pmod 2; \\  (n+m)\cdots(n+3)(n+1)\pi(n+2)(n+4)\cdots(n+m-1)& \mbox{if } m\equiv 1\pmod 2. \end{cases}
\end{equation}
For example, 
\[\chi_5(132)=86413257,\quad\text{and}\quad \chi_6(132)=864132579.\]

\begin{theorem}\label{Thm12}
Preserving the preceding notation, let \[\mathcal A=\bigcup_{m\geq 0}\{\chi_m(1),\chi_m(12),\chi_m(1423),\chi_m(2143)\}.\]
Let $\tau^{(1)},\tau^{(2)},\ldots$ be a (possibly empty) list of patterns taken from the set $\mathcal A$, and let $\tau'^{(i)}=\swu(\tau^{(i)})$ for all $i$. The permutation classes $\Av(231,\tau^{(1)},\tau^{(2)},\ldots)$ and $\Av(132,\tau'^{(1)},\tau'^{(2)},\ldots)$ are postorder Wilf equivalent. 
\end{theorem}

\begin{proof}
We know that $\swu:\Av(231)\to\Av(132)$ is a bijection. We will show that \[\swu(\Av(231,\tau^{(1)},\tau^{(2)},\ldots))=\Av(132,\tau'^{(1)},\tau'^{(2)},\ldots).\] This will allow us to define \[\eta:P^{-1}(\Av(231,\tau^{(1)},\tau^{(2)},\ldots))\to P^{-1}(\Av(132,\tau'^{(1)},\tau'^{(2)},\ldots))\] by \[\eta(T)=\eta_{P(T)}(T),\] where $\eta_{P(T)}:P^{-1}(P(T))\to P^{-1}(\swu(P(T)))$ is the skeleton-preserving bijection from Proposition~\ref{Prop2}. It will then follow that $\eta$ is a skeleton-preserving bijection, which will complete the proof. By taking intersections, we find that it suffices to prove that 
\begin{equation}\label{Eq17}
\swu(\Av(231,\tau))=\Av(132,\swu(\tau))
\end{equation} for all $\tau\in\mathcal A$. 

Choose $\tau\in\mathcal A$, and write $\tau=\chi_m(\mu)$ for some $m\geq 0$ and $\mu\in\{1,12,1423,2143\}$. We induct on $m$. Suppose that $m\geq 1$ and that we have already proven the equality $\swu(\Av(231,\chi_{m-1}(\mu)))=\Av(132,\swu(\chi_{m-1}(\mu)))$. If $m$ is even, then $\chi_{m-1}(\mu)$ is sum indecomposable and $\tau=\chi_{m-1}(\mu)\oplus 1$, so we can use Proposition \ref{Prop3} to see that \eqref{Eq17} holds. If $m$ is odd, then $\chi_{m-1}(\mu)$ is skew indecomposable and $\tau=1\ominus\chi_{m-1}(\mu)$, so we can use Proposition \ref{Prop3} to see that \eqref{Eq17} holds in this case as well. This completes the proof of the inductive step, so it remains to prove the base case in which $m=0$. In other words, we need to prove \eqref{Eq17} when $\tau=\mu\in\{1,12,1423,2143\}$. This is easy if $\mu\in\{1,12\}$, so we may assume $\mu\in\{1423,2143\}$. 

We wish to show that $\swu(\Av_n(231,\mu))=\Av_n(132,\swu(\mu))$ for all $n\geq 0$. This is trivial if $n\leq 2$, so we may assume $n\geq 3$ and induct on $n$. Fix $\pi\in\Av_n(231,\mu)$ and $\sigma\in\Av_n(132,\swu(\mu))$. Our goal is to show that $\swu(\pi)\in\Av_n(132,\swu(\mu))$ and $\swu^{-1}(\sigma)\in\Av_n(231,\mu)$. We know that $\swu(\pi)$ avoids $132$ and $\swu^{-1}(\sigma)$ avoids $231$, so we need only prove that $\swu(\pi)$ avoids $\swu(\mu)$ and that $\swu^{-1}(\sigma)$ avoids $\mu$.   

Let us write 
\begin{equation}\label{Eq11}
\pi=L\oplus(1\ominus R)\quad \text{and}\quad\sigma=(\widehat L\oplus 1)\ominus\widehat R
\end{equation} so that 
\begin{equation}\label{Eq12}
\swu(\pi)=(\swu(L)\oplus 1)\ominus \swu(R)
\quad\text{and}\quad\swu^{-1}(\sigma)=\swu^{-1}(\widehat L)\oplus(1\ominus\swu^{-1}(\widehat R)).
\end{equation} 
Because $L,R\in\Av(231,\mu)$, it follows by induction that $\swu(L)$ and $\swu(R)$ avoid $\swu(\mu)$. A similar argument shows that $\swu^{-1}(\widehat L)$ and $\swu^{-1}(\widehat R)$ avoid $\mu$. We now consider cases based on whether $\mu$ is $1423$ or $2143$. 

First, suppose $\mu=1423$. In this case, $\swu(\mu)=3412$. Assume by way of contradiction that $\swu(\pi)$ contains $3412$. Because $\swu(R)$ avoids $3412$, it follows from \eqref{Eq12} that $\swu(L)$ is nonempty. Hence, $L$ is nonempty. Because $\swu(L)$ avoids $3412$, it follows from \eqref{Eq12} that $\swu(R)$ has at least one ascent. This tells us that $\swu(R)$ is not a strictly decreasing permutation, which means that $R$ is not strictly decreasing either. Hence, $R$ has an ascent. It is now immediate from \eqref{Eq11} that $\pi$ contains $1423$, which is a contradiction. Now assume $\swu^{-1}(\sigma)$ contains $1423$. Because $\swu^{-1}(\widehat R)$ avoids $1423$, we can use \eqref{Eq12} to see that $\swu^{-1}(\widehat L)$ is nonempty. Similarly, $\swu^{-1}(\widehat R)$ has an ascent because $\swu^{-1}(\widehat L)$ avoids $1423$. As a consequence, $\widehat L$ is nonempty, and $\widehat R$ has an ascent. It is now immediate from \eqref{Eq11} that $\sigma$ contains $3412$, which is a contradiction. This handles the case in which $\mu=1423$. 

For the second case, suppose $\mu=2143$. In this case, $\swu(\mu)=3241$. Assume by way of contradiction that $\swu(\pi)$ contains $3241$. Using \eqref{Eq12} and the fact that $\swu(R)$ and $\swu(L)$ avoid $3241$, we deduce that $\swu(L)$ has a descent and $\swu(R)$ is nonempty. This implies that $L$ has a descent and $R$ is nonempty. It follows from \eqref{Eq11} that $\pi$ contains $2143$, which is a contradiction. A similar argument shows that $\swu^{-1}(\sigma)$ avoids $2143$. This handles the case in which $\mu=2143$. 
\end{proof}

\begin{remark}\label{Rem3}
Even if we just take the sequence $\tau^{(1)},\tau^{(2)},\ldots$ in Theorem \ref{Thm12} to be empty, we obtain the new result that $\Av(231)$ and $\Av(132)$ are postorder Wilf equivalent. The vast generality of Theorem \ref{Thm12} comes from the vast generality inherent in the definition of postorder Wilf equivalence coupled with the large size of the set $\mathcal A$. We could obtain infinitely many examples of postorder Wilf equivalence even if $\mathcal A$ was replaced by $\bigcup_{m\geq 0}\{\chi_m(1)\}$; the other elements of $\mathcal A$ just yield more examples!  \hspace*{\fill}$\lozenge$ 
\end{remark}

\begin{theorem}\label{Thm14}
Let \[\mathcal A=\bigcup_{m\geq 0}\{\chi_m(1),\chi_m(12),\chi_m(1423),\chi_m(2143)\}.\]
Let $\tau^{(1)},\tau^{(2)},\ldots$ be a (possibly empty) list of patterns taken from the set $\mathcal A$, and let $\tau'^{(i)}=\swu(\tau^{(i)})$ for all $i$. The permutation classes $\Av(231,\tau^{(1)},\tau^{(2)},\ldots)$ and $\Av(132,\tau'^{(1)},\tau'^{(2)},\ldots)$ are strongly fertility Wilf equivalent. 
\end{theorem}

\begin{proof}

For $\pi\in\Av_n(231)$, we can transform a valid hook configuration $\mathcal H=(H_1,\ldots,H_m)$ of $\pi$ into a valid hook configuration $\widehat\swu(\mathcal H)$ of $\swu(\pi)$ (see Figure \ref{Fig13}). The plot of $\swu(\pi)$ is obtained by vertically sliding the points in the plot of $\pi$. During this sliding process, we simply keep all the hooks in $\mathcal H$ attached to the same points to obtain $\widehat\swu(\mathcal H)$. In order to state this more precisely, let $(i_u,\pi_{i_u})$ and $(j_u,\pi_{j_u})$ be the southwest and northeast endpoints of $H_u$, respectively. We let $\widehat\swu(\mathcal H)=(H_1',\ldots,H_m')$, where $H_u'$ is the hook with southwest endpoint $(i_u,\swu(\pi)_{i_u})$ and northeast endpoint $(j_u,\swu(\pi)_{j_u})$. 

\begin{figure}[h]
\begin{center}
\includegraphics[width=.825\linewidth]{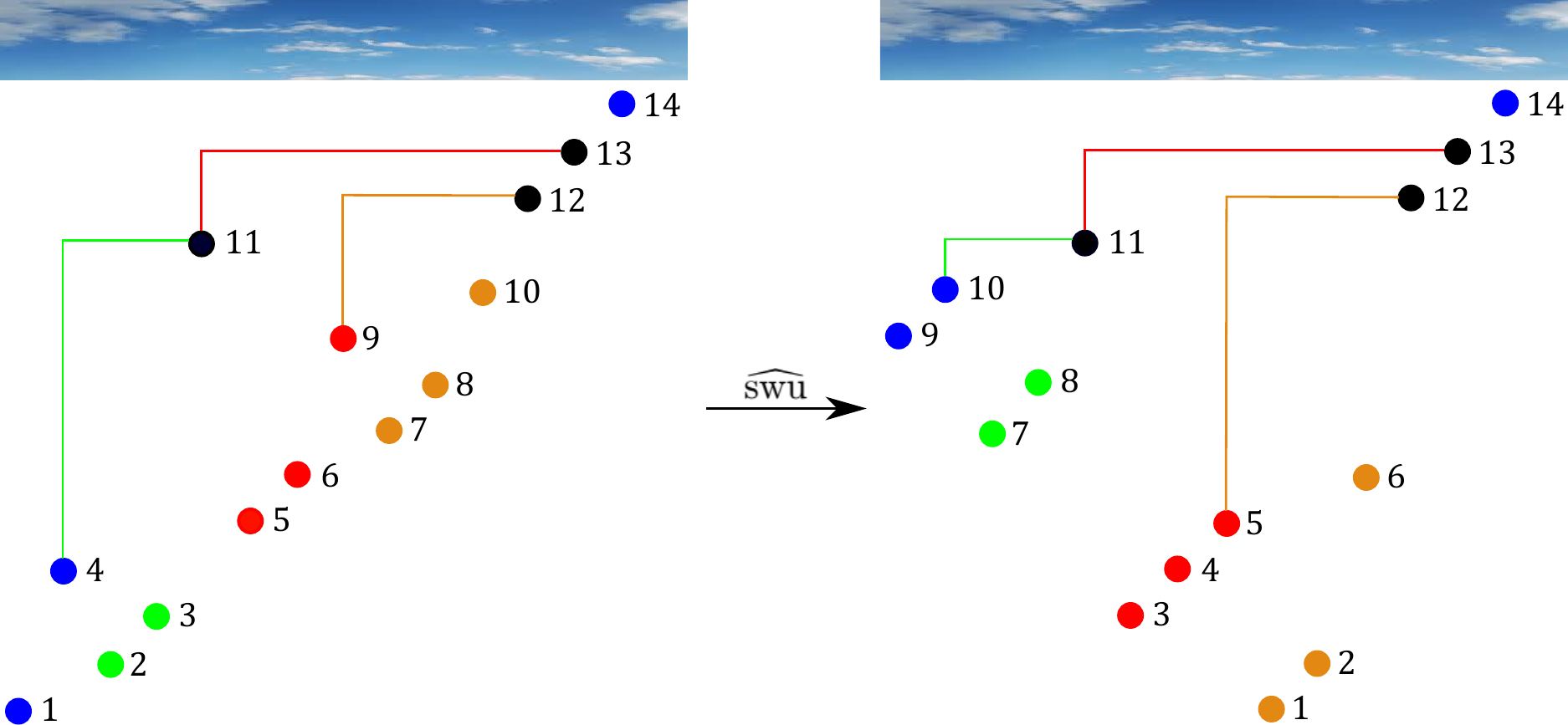}
\caption{Each valid hook configuration $\mathcal H$ of $\pi$ corresponds to a valid hook configuration $\widehat\swu(\mathcal H)$ of $\swu(\pi)$. We have drawn the induced colorings to show that the transformation does not change the induced valid composition. Indeed, both of these valid hook configurations induce the valid composition $(3,2,3,3)$.}
\label{Fig13}
\end{center}  
\end{figure}

We saw in the proof of Theorem \ref{Thm12} that \[\swu(\Av(231,\tau^{(1)},\tau^{(2)},\ldots))=\Av(132,\tau'^{(1)},\tau'^{(2)},\ldots).\] Using Lemma \ref{Lem1} and Definitions \ref{Def9} and \ref{Def8}, one can verify that $\widehat\swu$ gives a type-preserving bijection from $\VHC(\pi)$ to $\VHC(\swu(\pi))$ for every $\pi\in\Av(231,\tau^{(1)},\tau^{(2)},\ldots)$. Therefore, it also gives a type-preserving bijection \[\VHC(\Av(231,\tau^{(1)},\tau^{(2)},\ldots))\to\VHC(\Av(132,\tau'^{(1)},\tau'^{(2)},\ldots)).\qedhere\]  
\end{proof}

In the specific case of the following corollary in which the sequence $\tau^{(1)},\tau^{(2)},\ldots$ is empty, we recover the joint equidistribution result of Bouvel and Guibert \cite{Bouvel} mentioned in the introduction and also find several new statistics that are jointly equidistributed on $s^{-1}(\Av_n(231))$ and $s^{-1}(\Av_n(132))$. In the case in which the sequence consists of the single pattern $\tau^{(1)}=312$, we reprove and greatly generalize Theorem 10.1 from \cite{DefantClass}.  

\begin{corollary}\label{Cor3}
Let $\tau^{(1)},\tau^{(2)},\ldots$ and $\tau'^{(1)},\tau'^{(2)},\ldots$ be sequences as in the statement of Theorem~\ref{Thm12}. For every $n\geq 0$, the statistic $\zeil$ and all skeletal statistics are jointly equidistributed on $s^{-1}(\Av_n(231,\tau^{(1)},\tau^{(2)},\ldots))$ and $s^{-1}(\Av_n(132,\tau'^{(1)},\tau'^{(2)},\ldots))$. 
\end{corollary}

\begin{proof}
In the proof of Theorem \ref{Thm12}, we saw that there is a skeleton-preserving bijection \[\eta:P^{-1}(\Av(231,\tau^{(1)},\tau^{(2)},\ldots))\to P^{-1}(\Av(132,\tau'^{(1)},\tau'^{(2)},\ldots))\] given by \[\eta(T)=\eta_{P(T)}(T),\] where $\eta_{P(T)}:P^{-1}(P(T))\to P^{-1}(\swu(P(T)))$ is the skeleton-preserving bijection from Proposition~\ref{Prop2}. Following the proof of Proposition \ref{Prop1}, we see that $\eta$ restricts to a skeleton-preserving bijection \[\widetilde\eta:P^{-1}(\Av_n(231,\tau^{(1)},\tau^{(2)},\ldots))\cap\DPT^{(2)}\to P^{-1}(\Av_n(132,\tau'^{(1)},\tau'^{(2)},\ldots))\cap\DPT^{(2)}.\] As in the proof of Proposition \ref{Prop1}, we note that the map $\eta^{\#}=I\circ\widetilde\eta\circ I^{-1}$ yields a bijection from $s^{-1}(\Av_n(231,\tau^{(1)},\tau^{(2)},\ldots))$ to $s^{-1}(\Av_n(132,\tau'^{(1)},\tau'^{(2)},\ldots))$ for every $n\geq 0$. We need to show that $f(\sigma)=f(\eta^{\#}(\sigma))$ for every $\sigma\in s^{-1}(\Av_n(231,\tau^{(1)},\tau^{(2)},\ldots))$ and every statistic $f$ that is either $\zeil$ or is skeletal. If $f$ is skeletal, then this follows immediately from the fact that $\widetilde\eta$ is skeleton-preserving. We are left to consider the case $f=\zeil$. 

Choose $\sigma\in s^{-1}(\Av_n(231,\tau^{(1)},\tau^{(2)},\ldots))$, and let $T=I^{-1}(\sigma)$. According to \eqref{Eq6}, $s(\sigma)=P(T)$. We have \[\widetilde\eta(T)=\eta(T)=\eta_{P(T)}(T)\in P^{-1}(\swu(P(T)))=P^{-1}(\swu(s(\sigma))),\] so \[s(\eta^{\#}(\sigma))=s\circ I\circ\widetilde\eta\circ I^{-1}(\sigma)=P\circ \widetilde\eta\circ I^{-1}(\sigma)=P(\widetilde\eta(T))=\swu(s(\sigma)).\] As a consequence, we can use Lemma \ref{Lem1} to see that $\tl(s(\eta^{\#}(\sigma)))=\tl(\swu(s(\sigma)))=\tl(s(\sigma))$. It is not difficult to show that the statistic $\rmax$ is skeletal, so we know from above that $\rmax(\eta^{\#}(\sigma))=\rmax(\sigma)$. We now invoke Lemma \ref{Lem4} to find that \[\zeil(\eta^{\#}(\sigma))=\min\{\rmax(\eta^{\#}(\sigma)),\tl(s(\eta^{\#}(\sigma)))\}=\min\{\rmax(\sigma),\tl(s(\sigma))\}=\zeil(\sigma). \qedhere\]
\end{proof}

We now turn our attention to the map $\swl$. The proof of the following proposition makes use of a general procedure that allows us to decompose valid hook configurations. This procedure (phrased differently) has been crucial for enumerating $3$-stack-sortable permutations and stack-sorting preimages of  permutation classes \cite{DefantCounting,DefantEnumeration}. Let $\pi=\pi_1\cdots\pi_n$ be a permutation, and let $H$ be a hook of $\pi$ with southwest endpoint $(i,\pi_i)$ and northeast endpoint $(j,\pi_j)$. Let us assume that $j$ is larger than every descent of $\pi$. The hook $H$ separates $\pi$ into two parts. One part, which we call the \emph{$H$-unsheltered subpermutation of $\pi$} and denote by $\pi_U^H$, is $\pi_1\cdots\pi_i\pi_{j+1}\cdots\pi_n$. The other part, which we call the \emph{$H$-sheltered subpermutation of $\pi$} and denote by $\pi_S^H$, is $\pi_{i+1}\cdots\pi_{j-1}$. Note that the entry $\pi_j$ does not appear in either of these two parts. 

This decomposition of $\pi$ into the $H$-unsheltered and $H$-sheltered subpermutations provides a useful decomposition of valid hook configurations of $\pi$ that include $H$. Denote the set of such valid hook configurations by $\VHC^H(\pi)$. We have maps \[\varphi_U^H:\VHC^H(\pi)\to\VHC(\pi_U^H)\quad\text{and}\quad\varphi_S^H:\VHC^H(\pi)\to\VHC(\pi_S^H).\] Rather than describe these maps in words, we find it more instructive to give an illustrative example in Figure \ref{Fig14}. The following lemma provides useful information about these maps. We use Figure \ref{Fig14} as a substitute for the proof of this lemma.  

\begin{figure}[h]
\begin{center}
\includegraphics[width=.85\linewidth]{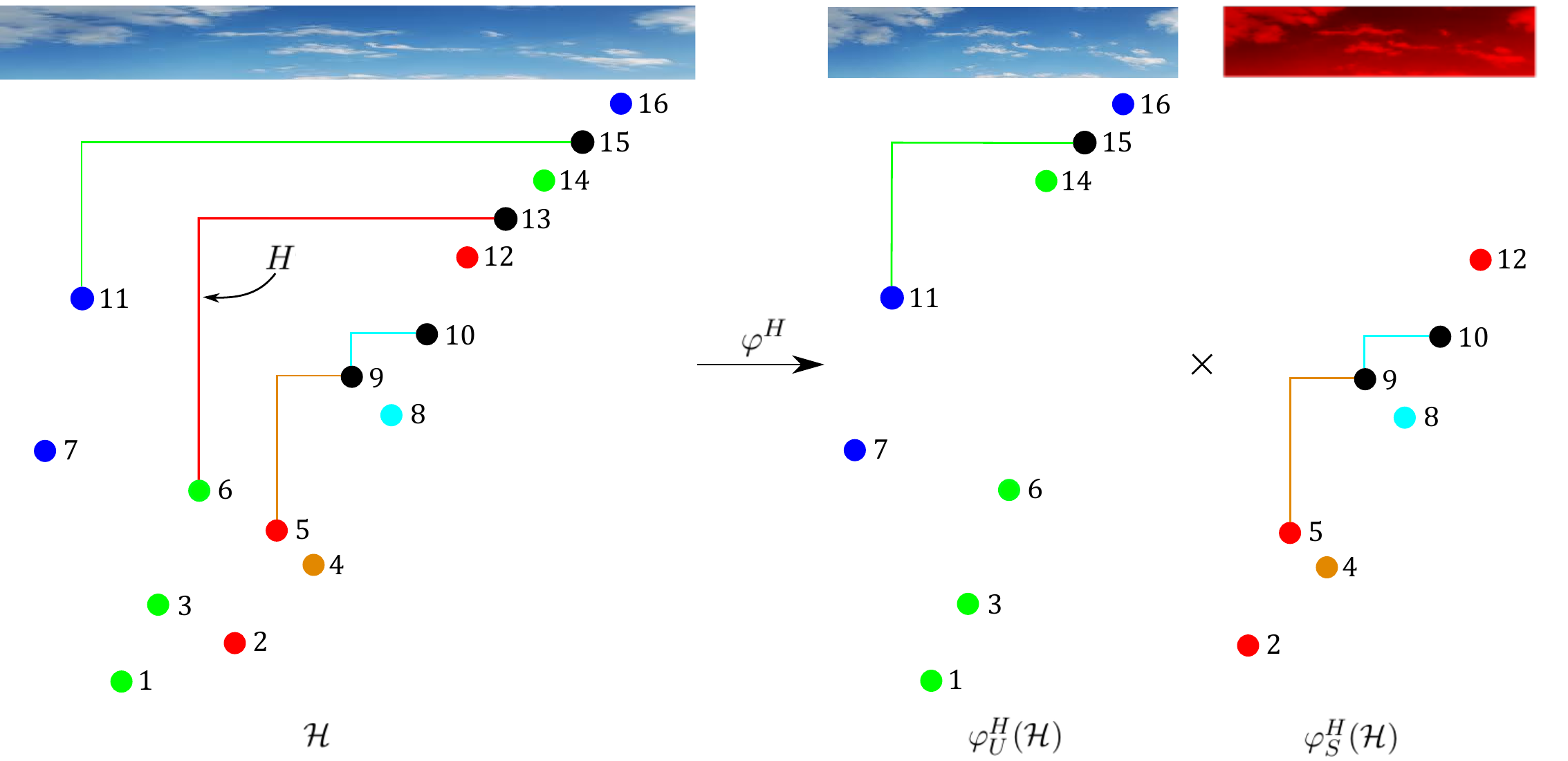}
\caption{An illustration of the maps $\varphi_U^H$ and $\varphi_S^H$ from Lemma \ref{Lem2}. Notice that the hook $H$ in $\mathcal H$ ``becomes" the sky in $\varphi_S^H(\mathcal H)$; this is why we colored that sky red instead of the usual color blue. The valid composition induced by $\mathcal H$ is $(3,4,3,1,1)$, and the valid compositions induced by $\varphi_U^H(\mathcal H)$ and $\varphi_S^H(\mathcal H)$ are $(3,4)$ and $(3,1,1)$, respectively.}
\label{Fig14}
\end{center}  
\end{figure}

\begin{lemma}\label{Lem2}
Let $\pi=\pi_1\cdots\pi_n$ be a permutation with descents $d_1<\cdots<d_k$. Let $H$ be a hook of $\pi$ with southwest endpoint $(d_i,\pi_{d_i})$ and northeast endpoint $(j,\pi_j)$, and assume $j>d_k$. Let $\VHC^H(\pi)$ be the set of valid hook configurations of $\pi$ that include the hook $H$. The maps \[\varphi_U^H:\VHC^H(\pi)\to\VHC(\pi_U^H)\quad\text{and}\quad\varphi_S^H:\VHC^H(\pi)\to\VHC(\pi_S^H)\] are such that the map \[\varphi^H:\VHC^H(\pi)\to\VHC(\pi_U^H)\times\VHC(\pi_S^H)\] given by $\varphi^H(\mathcal H)=(\varphi_U^H(\mathcal H),\varphi_S^H(\mathcal H))$ is a bijection. If $\mathcal H\in\VHC^H(\pi)$ induces the valid composition $(q_0,\ldots,q_k)$, then the valid compositions induced by $\varphi_U^H(\mathcal H)$ and $\varphi_S^H(\mathcal H)$ are $(q_0,\ldots,q_{i-1})$ and $(q_i,\ldots,q_k)$, respectively. 
\end{lemma}

In the following proposition, recall the definition of a type-preserving map from the paragraph preceding Definition \ref{Def10}. Also, recall Remark \ref{Rem7}, which tells us how to define $\swl$ on $132$-avoiding permutations that are not necessarily normalized. 

\begin{proposition}\label{Prop4}
For every permutation $\pi$ that avoids $132$ and $3412$, there is a type-preserving bijection \[\theta_\pi:\VHC(\pi)\to\VHC(\swl(\pi)).\]   
\end{proposition}

\begin{proof}
In order to prove the proposition for a permutation $\pi$, it suffices to prove the proposition for the normalization of $\pi$; indeed, it is then easy to ``unnormalize" the relevant permutations and valid hook configurations. Thus, we may assume $\pi\in\Av_n(132,3412)$. To ease notation, let $\pi'=\pi'_1\cdots\pi'_n=\swl(\pi)$. The proof is by induction on $n$. If $n\leq 2$, then the conclusion is obvious because $\pi=\pi'$. Thus, we may assume that $n\geq 3$ and that $\pi\neq \pi'$. Let $a=n-\tl(\pi)$. This means that $\pi_r=r$ for all $r\in\{a+1,\ldots,n\}$ and that $\pi_{a}\neq a$. Lemma \ref{Lem1} tells us that $a=n-\tl(\pi')$. If $a=n$, then $\VHC(\pi)$ and $\VHC(\pi')$ are both empty (a permutation that has a valid hook configuration must end with its largest entry), so there is nothing to do. Thus, we may assume $a\leq n- 1$. It is straightforward to check that $a\geq 3$ because $\pi\neq\pi'$. We now consider two cases. 

First, assume $\pi_a=1$. The point $(a-1,\pi_{a-1})$ must be a descent top of the plot of $\pi$, so it follows from Definition \ref{Def9} that every valid hook configuration of $\pi$ has a hook with southwest endpoint $(a-1,\pi_{a-1})$. It is not difficult to check that $\pi_r'=\pi_r$ for all $r\in\{a-1,\ldots,n\}$, so every valid hook configuration of $\pi'$ has a hook with southwest endpoint $(a-1,\pi_{a-1})$. Fix $j\in\{a+1,\ldots,n\}$. Let $H$ be the hook of $\pi$ with southwest endpoint $(a-1,\pi_{a-1})$ and northeast endpoint $(j,j)$. Let $H'$ be the hook of $\pi'$ with southwest endpoint $(a-1,\pi_{a-1})$ and northeast endpoint $(j,j)$. Although $H$ and $H'$ have the same endpoints, we think of them as being distinct since they are hooks of different permutations. We will show that there is a type-preserving bijection \[\theta_\pi^H:\VHC^H(\pi)\to\VHC^{H'}(\pi').\] This will complete the proof in this case since we can simply combine the bijections $\theta_\pi^H$ for all possible choices of $H$ (i.e., all possible choices of $j$) to form the desired bijection $\theta_\pi$. 

Consider the $H$-unsheltered and $H$-sheltered subpermutations $\pi_U^H=\pi_1\cdots\pi_{a-1}\pi_{j+1}\cdots\pi_n$ and $\pi_S^H=\pi_a\cdots\pi_{j-1}=1(a+1)(a+2)\cdots(j-1)$. Similarly, consider the $H'$-unsheltered and $H'$-sheltered subpermutations $(\pi')_U^{H'}=\pi'_1\cdots\pi'_{a-1}\pi'_{j+1}\cdots\pi'_n$ and $(\pi')_S^{H'}=\pi'_a\cdots\pi'_{j-1}=1(a+1)(a+2)\cdots(j-1)=\pi_S^H$. It is straightforward to verify that $\swl(\pi_U^H)=(\pi')_U^{H'}$ and $\swl(\pi_S^H)=\pi_S^H=(\pi')_S^{H'}$. Therefore, we know by induction that there exist type-preserving bijections \[\theta_{\pi_U^H}:\VHC(\pi_U^H)\to\VHC((\pi')_U^{H'})\quad\text{and}\quad\theta_{\pi_S^H}:\VHC(\pi_S^H)\to\VHC((\pi')_S^{H'}).\] Invoking Lemma \ref{Lem2}, we find the maps 
\[\VHC^H(\pi)\xrightarrow{\varphi^H}\VHC(\pi_U^H)\times\VHC(\pi_S^H)\xrightarrow{\theta_{\pi_U^H}\times\theta_{\pi_S^H}}\VHC((\pi')_U^{H'})\times\VHC((\pi')_S^{H'})\xrightarrow{(\varphi^{H'})^{-1}}\VHC^{H'}(\pi').\] The composite map $(\varphi^{H'})^{-1}\circ(\theta_{\pi_U^H}\times\theta_{\pi_S^H})\circ\varphi^H$ is a bijection. Using the last part of Lemma \ref{Lem2} along with the fact that the maps $\theta_{\pi_U^H}$ and $\theta_{\pi_S^H}$ are type-preserving, we find that this composite map is also type-preserving. Hence, we can set $\theta_\pi^H=(\varphi^{H'})^{-1}\circ(\theta_{\pi_U^H}\times\theta_{\pi_S^H})\circ\varphi^H$ to complete the proof in the case $\pi_a=1$. Figure \ref{Fig4} illustrates the construction of $\theta_\pi^H$ in this case. 

\begin{figure}[h]
\begin{center}
\includegraphics[width=.5\linewidth]{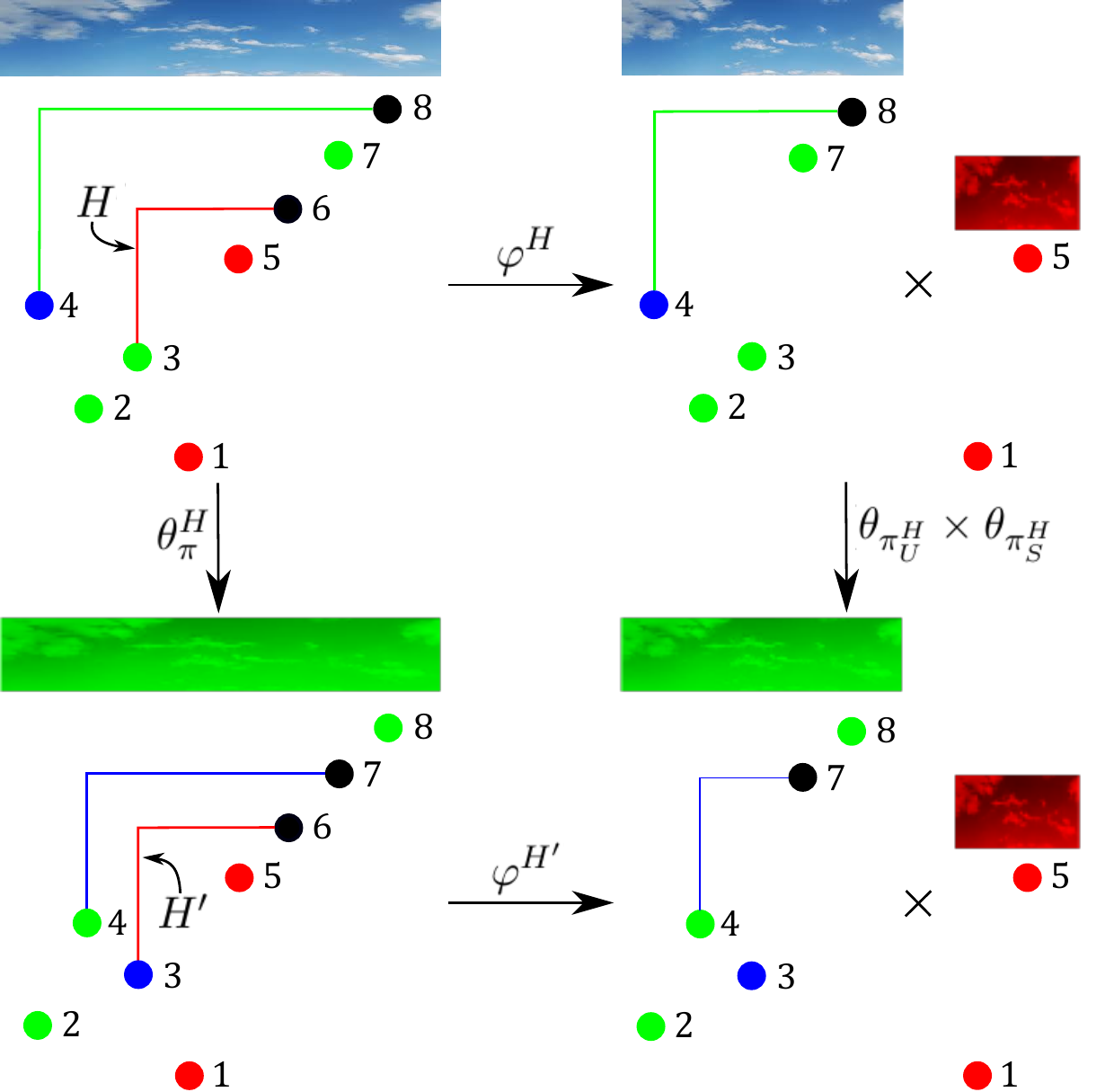}
\caption{An example illustrating the construction of $\theta_\pi^H$ in the case in which $\pi_a=1$. The transformation $\theta_{\pi_U^H}\times\theta_{\pi_S^H}$ is defined recursively; we omit the steps transforming the top right section into the bottom right section in this example. This transformation sometimes interchanges a sky with a hook, which is why one of the skies in the bottom right section is green.}
\label{Fig4}
\end{center}  
\end{figure}

For the second case, assume $\pi_a=b\geq 2$. Because $\pi$ avoids $132$, we can write $\pi=(\lambda\ominus(\mu\oplus 1))\oplus\Id_{n-a}$, where $\Id_{n-a}=123\cdots (n-a)$ is the identity permutation in $S_{n-a}$, $\mu\in S_{b-1}$, and $\lambda\in S_{a-b}$. Note that $\mu$ is nonempty because $b\geq 2$. Since $n-a$ is the tail length of $\pi$, $b\leq a-1$. This means that $\lambda$ is also nonempty. The permutation $\lambda$ cannot have an ascent because $\pi$ avoids $3412$. Therefore, $\lambda=\delta_{a-b}$ is the decreasing permutation in $S_{a-b}$. This allows us to write 
\begin{equation}\label{Eq21}
\pi=(\delta_{a-b}\ominus(\mu\oplus 1))\oplus\Id_{n-a}.
\end{equation} It is straightforward to check that \[\pi'=\swl(\mu)\oplus\delta_{a-b+1}\oplus\Id_{n-a},\] where $\delta_{a-b+1}$ is the decreasing permutation in $S_{a-b+1}$. The point $(a-b,b+1)$ is a descent top of the plot of $\pi$, and the point $(b,a)$ is a descent top of the plot of $\pi'$. Fix $\ell\in[n-a]$. Let $H$ be the hook of $\pi$ with southwest endpoint $(a-b,b+1)$ and northeast endpoint $(a+\ell,a+\ell)$. Let $H'$ be the hook of $\pi'$ with southwest endpoint $(b,a)$ and northeast endpoint $(n+1-\ell,n+1-\ell)$. We will show that there is a type-preserving bijection \[\theta_\pi^H:\VHC^H(\pi)\to\VHC^{H'}(\pi').\] This will complete the proof in this case since we can simply combine the bijections $\theta_\pi^H$ for all possible choices of $H$ (i.e., all possible choices of $\ell$) to form the desired bijection $\theta_\pi$. 

Consider the $H$-unsheltered and $H$-sheltered subpermutations \[\pi_U^H=\pi_1\cdots\pi_{a-b}\pi_{a+\ell+1}\cdots\pi_n=a(a-1)\cdots(b+1)(a+\ell+1)(a+\ell+2)\cdots n\] and \[\pi_S^H=\pi_{a-b+1}\cdots\pi_{a+\ell-1}=\mu\, b(a+1)(a+2)\cdots(a+\ell-1).\] Similarly, consider the $H'$-unsheltered and $H'$-sheltered subpermutations \[(\pi')_U^{H'}=\pi'_1\cdots\pi'_b\pi'_{n+2-\ell}\cdots\pi'_n=\swl(\mu)\,a(n+2-\ell)(n+3-\ell)\cdots n\] and \[(\pi')_S^{H'}=\pi'_{b+1}\cdots\pi'_{n-\ell}=(a-1)(a-2)\cdots b(a+1)(a+2)\cdots(n-\ell).\] One can verify that $\swl(\pi_U^H)$ and $(\pi')_S^{H'}$ have the same normalization, so there is a natural type-preserving bijection $\omega_1:\VHC(\swl(\pi_U^H))\to\VHC((\pi')_S^{H'})$. We know by induction that there is a type-preserving bijection $\theta_{\pi_U^H}:\VHC(\pi_U^H)\to\VHC(\swl(\pi_U^H))$, so we obtain a type-preserving bijection $\psi_1=\omega_1\circ\theta_{\pi_U^H}:\VHC(\pi_U^H)\to\VHC((\pi')_S^{H'})$. We can also check that $\swl(\pi_S^H)$ and $(\pi')_U^{H'}$ have the same normalization, so a similar argument produces type-preserving bijections $\omega_2:\VHC(\swl(\pi_S^H))\to\VHC((\pi')_U^{H'})$ and $\psi_2=\omega_2\circ\theta_{\pi_S^H}:\VHC(\pi_S^H)\to\VHC((\pi')_U^{H'})$. 

Invoking Lemma \ref{Lem2}, we find the maps 
\[\VHC^H(\pi)\xrightarrow{\varphi^H}\VHC(\pi_U^H)\times\VHC(\pi_S^H)\xrightarrow{\psi_1\times\psi_2}\VHC((\pi')_S^{H'})\times\VHC((\pi')_U^{H'})\] \[\xrightarrow{\delta}\VHC((\pi')_U^{H'})\times\VHC((\pi')_S^{H'})\xrightarrow{(\varphi^{H'})^{-1}}\VHC^{H'}(\pi'),\] where $\delta$ is defined by $\delta(\mathcal H_1,\mathcal H_2)=(\mathcal H_2,\mathcal H_1)$. The composite map $(\varphi^{H'})^{-1}\circ\delta\circ(\psi_1\times\psi_2)\circ\varphi^H$ is a bijection. Using the last part of Lemma \ref{Lem2} along with the fact that the maps $\psi_1$ and $\psi_2$ are type-preserving, we find that this composite map is also type-preserving. Hence, we can set $\theta_\pi^H=(\varphi^{H'})^{-1}\circ\delta\circ(\psi_1\times\psi_2)\circ\varphi^H$ to complete the proof of the case in which $\pi_a\geq 2$. Figure \ref{Fig5} illustrates the construction of $\theta_\pi^H$ in this case. 
\end{proof}   
\begin{figure}[h]
\begin{center}
\includegraphics[width=152mm]{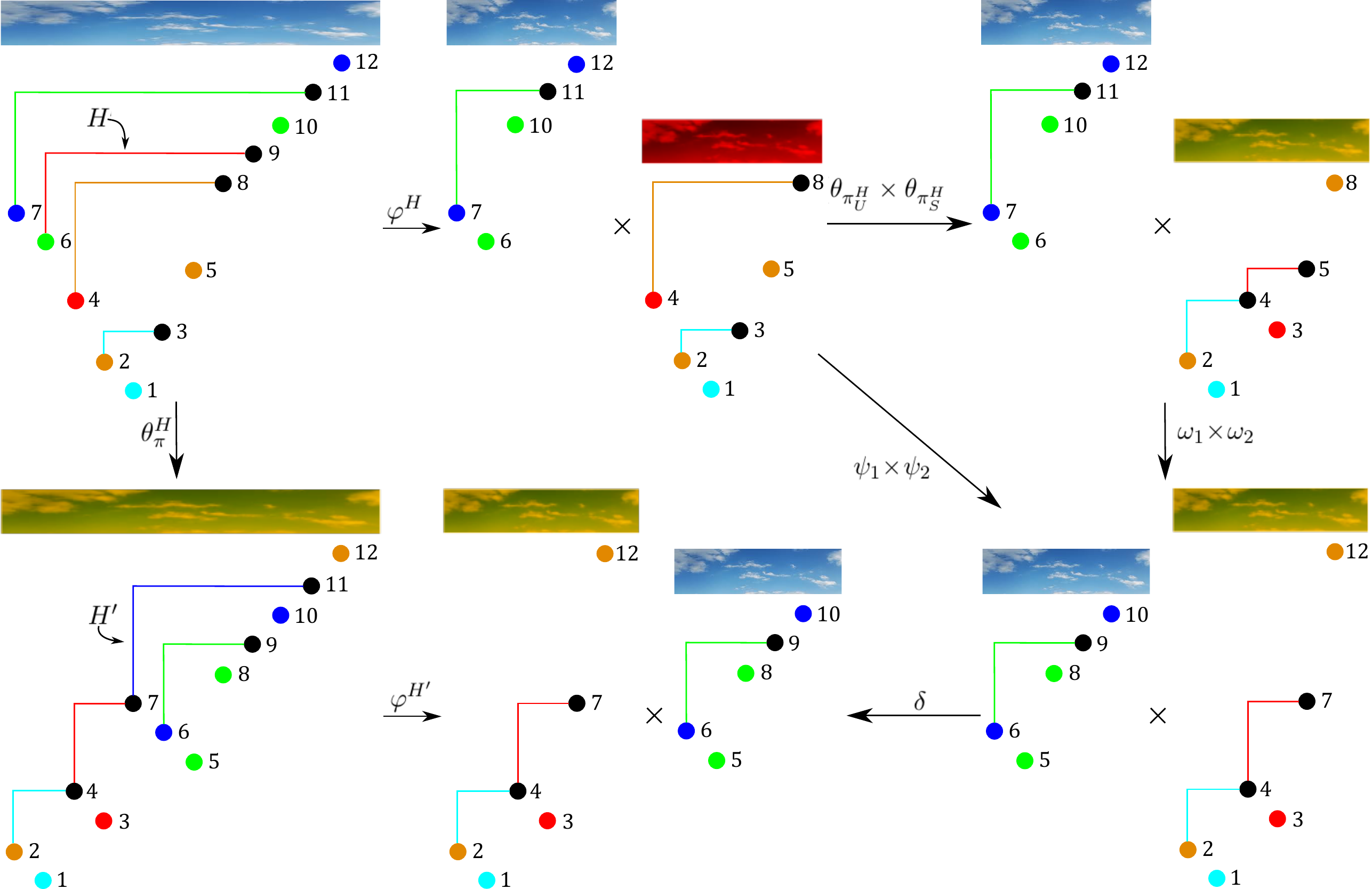}
\caption{An example illustrating the construction of $\theta_\pi^H$ in the case in which $\pi_a\geq 2$. The transformation $\theta_{\pi_U^H}\times\theta_{\pi_S^H}$ is defined recursively; we omit the steps transforming the top middle section into the top right section in this example. This transformation sometimes interchanges a sky with a hook, which is why one of the skies in the top right section is orange.}
\label{Fig5}
\end{center}  
\end{figure}

Propositions \ref{Prop3} and \ref{Prop4} allow us to produce several examples of strong fertility Wilf equivalences. The following theorem exhibits infinitely many such examples, but there could certainly be others. The proof below does not directly cite Proposition \ref{Prop3}, but it does cite the proof of Theorem \ref{Thm12}, which, in turn, makes heavy use of Proposition \ref{Prop3}. Recall the definition of $\rot$ from the beginning of this section and the definition of $\chi_m$ from \eqref{Eq22}. For $\pi\in S_n$, let $\widetilde\chi_m(\pi)=\rot^{-1}(\chi_m(\rot(\pi)))$. 

\begin{theorem}\label{Thm13}
Preserving the preceding notation, let \[\mathcal B=\bigcup_{m\geq 0}\{\widetilde\chi_m(1),\widetilde\chi_m(21),\widetilde\chi_m(2431)\}.\]
Let $\tau^{(1)},\tau^{(2)},\ldots$ be a (possibly empty) list of patterns taken from the set $\mathcal B$, and let $\tau'^{(i)}=\swl(\tau^{(i)})$ for all $i$. The permutation classes \[\Av(132,3412,\tau^{(1)},\tau^{(2)},\ldots)\quad\text{and}\quad\Av(312,1342,\tau'^{(1)},\tau'^{(2)},\ldots)\] are strongly fertility Wilf equivalent. 
\end{theorem} 

\begin{proof}
By Lemma \ref{Lem1}, we know that $\swl:\Av(132)\to\Av(312)$ is a bijection. We will show that \[\swl(\Av(132,3412,\tau^{(1)},\tau^{(2)},\ldots))=\Av(312,1342,\tau'^{(1)},\tau'^{(2)},\ldots).\] This will allow us to define \[\theta:\VHC(\Av(132,3412,\tau^{(1)},\tau^{(2)},\ldots))\to \VHC(\Av(132,1342,\tau'^{(1)},\tau'^{(2)},\ldots))\] by \[\theta(\mathcal H)=\theta_{\pi}(\mathcal H),\] where $\pi$ is the underlying permutation of $\mathcal H$ and $\theta_{\pi}:\VHC(\pi)\to \VHC(\swl(\pi))$ is the type-preserving bijection from Proposition \ref{Prop4}. The map $\theta$ will be a type-preserving bijection, so this will complete the proof. By taking intersections, we find that it suffices to prove that 
\[
\swl(\Av(132,3412,\tau))=\Av(312,1342,\swl(\tau))
\] for all $\tau\in \mathcal B$. 

Fix $\tau\in\mathcal B$. It is straightforward to check that $\rot(\tau)\in \mathcal A$, where $\mathcal A$ is the set in Theorem \ref{Thm12}. We saw in the proof of Theorem \ref{Thm12} that \[\swu(\Av(231,\rot(\tau)))=\Av(132,\swu(\rot(\tau))).\] We also saw that $\swu(\Av(231,2143))=\Av(132,3241)$, so \[\swu(\Av(231,2143,\rot(\tau)))=\Av(132,3241,\swu(\rot(\tau))).\] Since $\swl=\rot^{-1}\circ\swu\circ\rot$ by definition, we have \[\swl(\Av(132,3412,\tau))=\rot^{-1}\circ\swu(\Av(231,2143,\rot(\tau)))\] \[=\rot^{-1}(\Av(132,3241,\swu(\rot(\tau))))=\Av(312,1342,\swl(\tau)). \qedhere\] 
\end{proof}

\begin{corollary}\label{Cor2}
For every $n\geq 1$, the statistics $\des$ and $\peak$ are jointly equidistributed on $s^{-1}(\Av_n(132,231))$ and $s^{-1}(\Av_n(231,312))$. In particular, $\Av(132,231)$ and $\Av(231,312)$ are fertility Wilf equivalent. 
\end{corollary}

\begin{proof}
Take the sequence $\tau^{(1)},\tau^{(2)},\ldots$ in Theorem \ref{Thm13} to consist of the single pattern $231=\widetilde\chi_2(1)$ to find that the permutation classes $\Av(132,3412,231)=\Av(132,231)$ and \linebreak $\Av(312,1342,231)=\Av(231,312)$ are strongly fertility Wilf equivalent. Now use Proposition \ref{Prop5}.  
\end{proof}

\begin{remark}\label{Rem5}
Corollary \ref{Cor2}, which follows from a very special case of Theorem \ref{Thm13}, settles the conjecture from \cite{DefantClass} stating that $\Av(132,231)$ and $\Av(231,312)$ are fertility Wilf equivalent. If we were to replace the pattern $3412$ in Theorem \ref{Thm13} with $231$, we could simplify the proof of that theorem by immediately excluding the first of the two cases (the case $\pi_a=1$). This would yield a weaker theorem, but even this weaker theorem would suffice to prove the conjecture from \cite{DefantClass} and to produce infinitely many examples of strong fertility Wilf equivalence. We also could obtain infinitely many examples of strong fertility Wilf equivalence if $\mathcal B$ was replaced by $\bigcup_{m\geq 0}\{\widetilde \chi_m(1)\}$; the other elements of $\mathcal B$ just yield more examples! \hspace*{\fill}$\lozenge$ 
\end{remark}

For many of the possible choices of the sequence $\tau^{(1)},\tau^{(2)},\ldots$ in Theorem \ref{Thm13}, the permutation classes $\Av(132,3412,\tau^{(1)},\tau^{(2)},\ldots)$ and $\Av(312,1342,\tau'^{(1)},\tau'^{(2)},\ldots)$ are not postorder Wilf equivalent. In fact, they are not even binary postorder Wilf equivalent (recall the definition from Remark \ref{Rem4}) if $3124\in\Av(132,3412,\tau^{(1)},\tau^{(2)},\ldots)$. To see this, assume $3124\in\Av(132,3412,\tau^{(1)},\tau^{(2)},\ldots)$. We can check that $3124$ is the only permutation $\pi\in\Av_4(132,3412,\tau^{(1)},\tau^{(2)},\ldots)$ such that $P^{-1}(\pi)\neq\emptyset$ and $\swl(\pi)\neq\pi$. We have $\swl(3124)=1324$, so \[P^{-1}(\Av_4(312,1342,\tau'^{(1)},\tau'^{(2)},\ldots))\] \[=(P^{-1}(\Av_4(132,3412,\tau^{(1)},\tau^{(2)},\ldots))\setminus P^{-1}(3124))\cup P^{-1}(1324).\] If $\Av(132,3412,\tau^{(1)},\tau^{(2)},\ldots)$ and $\Av(312,1342,\tau'^{(1)},\tau'^{(2)},\ldots)$ were binary postorder Wilf equivalent, then there would be a skeleton-preserving bijection from \[P^{-1}(3124)\cap\DPT^{(2)}=I^{-1}(s^{-1}(3124))=\{I^{-1}(3412),I^{-1}(3421)\}\] to \[P^{-1}(1324)\cap\DPT^{(2)}=I^{-1}(s^{-1}(1324))=\{I^{-1}(3142),I^{-1}(1342)\}.\] This would imply that the statistics in \eqref{Eq8} other than $\zeil$ would be jointly equidistributed on $\{3412,3421\}$ and $\{3142,1342\}$. However, $\rmax$ is not even equidistributed on $\{3412,3421\}$ and $\{3142,1342\}$ because both permutations in $\{3142,1342\}$ have $2$ right-to-left maxima while  $3421$ has $3$.

It turns out that the statistic $\zeil$ \emph{is} equidistributed on $s^{-1}(\Av_n(132,3412,\tau^{(1)},\tau^{(2)},\ldots))$ and $s^{-1}(\Av_n(312,1342,\tau'^{(1)},\tau'^{(2)},\ldots))$ for all $n\geq 1$! This is surprising in light of Lemma \ref{Lem4} because we just saw that $\rmax$ is not necessarily equidistributed on these sets. Indeed, in our proof of Corollary \ref{Cor3}, we used the fact that the map $\eta^{\#}$ preserves $\rmax$ in order to deduce that it preserves $\zeil$. We will actually prove that $\des,\peak$, and $\zeil$ are jointly equidistributed on $s^{-1}(\Av_n(132,3412,\tau^{(1)},\tau^{(2)},\ldots))$ and $s^{-1}(\Av_n(312,1342,\tau'^{(1)},\tau'^{(2)},\ldots))$ for all $n\geq 1$; this gives much more than the mere equidistribution of $\zeil$ alone. 

For $\pi\in S_n$, let \[\mathcal Z_c(\pi)=\{\sigma\in s^{-1}(\pi):\zeil(\sigma)=c\}\quad\text{and}\quad\mathcal Z_{\geq c}(\pi)=\bigcup_{i\geq c}\mathcal Z_i(\pi).\] Let \[\mathcal Z_c^{a,b}(\pi)=\{\sigma\in \mathcal Z_c(\pi):\des(\sigma)=a,\peak(\sigma)=b\}\quad\text{and}\quad\mathcal Z_{\geq c}^{a,b}(\pi)=\bigcup_{i\geq c}\mathcal Z_i^{a,b}(\pi).\] Recall the tail length statistic $\tl$ from Section \ref{Sec:Stats}. For each $n\geq 1$ and $\lambda=\lambda_1\cdots\lambda_n\in S_n$, let $\lambda^*\in S_{n-1}$ be the normalization of $\lambda_1\cdots\lambda_{n-1}$. Let $D_\ell$ denote the set of permutations of which $\ell$ is a descent. 

\begin{lemma}\label{Lem3}
Preserve the notation from above. Let $\pi\in S_n$ for some $n\geq 3$. Suppose $a,b,c$ are nonnegative integers such that $1\leq c\leq \tl(\pi)-1$. The map $\sigma\mapsto\sigma^*$ induces bijections 
\begin{align*}
\mathcal Z_c^{a,b}(\pi)\setminus D_{n-1}&\to\mathcal Z_{\geq c}^{a,b}(\pi^*), \\ (\mathcal Z_c^{a,b}(\pi)\cap D_{n-1})\setminus D_{n-2}&\to\mathcal Z_{c-1}^{a-1,b-1}(\pi^*), \\ \mathcal Z_c^{a,b}(\pi)\cap D_{n-1}\cap D_{n-2}&\to\mathcal Z_{c-1}^{a-1,b}(\pi^*).
\end{align*}
\end{lemma}

\begin{proof}
The condition $1\leq c\leq\tl(\pi)-1$ forces $\tl(\pi)\geq 2$, so $\pi=\pi^*n$. Suppose $\sigma\in\mathcal Z_c(\pi)$. Since $\zeil(\sigma)=c$, we can write 
\[\sigma=\mu^{(0)}\,n\,\mu^{(1)}(n-1)\,\mu^{(2)}\cdots\mu^{(c-1)}(n-c+1)\,\mu^{(c)},\] where $n-c$ does not appear in the subpermutation $\mu^{(c)}$. By the definition of the stack-sorting map, we have \[\pi=s(\sigma)=s(\mu^{(0)})s(\mu^{(1)})\cdots s(\mu^{(c-1)})s(\mu^{(c)})(n-c+1)\cdots(n-1)n.\] Because $c<\tl(\pi)=\tl(s(\sigma))$, Lemma \ref{Lem4} tells us that $c=\rmax(\sigma)$. This means that $n,n-1,\ldots,n-c+1$ are the only right-to-left maxima of $\sigma$, so $\mu^{(c)}$ is empty. Thus, 
\begin{equation}\label{Eq19}
\sigma=\mu^{(0)}\,n\,\mu^{(1)}(n-1)\,\mu^{(2)}\cdots\mu^{(c-1)}(n-c+1),
\end{equation} and \[\pi^*=s(\mu^{(0)})s(\mu^{(1)})\cdots s(\mu^{(c-1)})(n-c+1)\cdots(n-1).\] Now, 
\begin{equation}\label{Eq20}
\sigma^*=\mu^{(0)}\,(n-1)\,\mu^{(1)}(n-2)\,\mu^{(2)}\cdots(n-c+1)\,\mu^{(c-1)},
\end{equation} so \[s(\sigma^*)=s(\mu^{(0)})s(\mu^{(1)})\cdots s(\mu^{(c-1)})(n-c+1)\cdots(n-1)=\pi^*.\] It is clear that $\zeil(\sigma^*)\geq c-1$, so $\sigma^*\in\mathcal Z_{\geq c-1}(\pi^*)$. 

We have seen that every element of $\mathcal Z_c(\pi)$ ends with the entry $n-c+1$, so the map $\mathcal Z_c(\pi)\to\mathcal Z_{\geq c-1}(\pi^*)$ given by $\sigma\mapsto\sigma^*$ is injective. To see that it is surjective, suppose $\lambda\in\mathcal Z_{\geq c-1}(\pi^*)$. Observe that $\lambda$ is of the form \[\widehat\mu^{(0)}\,(n-1)\,\widehat\mu^{(1)}(n-2)\,\widehat\mu^{(2)}\cdots(n-c+1)\,\widehat\mu^{(c-1)},\] where \[s(\widehat\mu^{(0)})s(\widehat\mu^{(1)})\cdots s(\widehat\mu^{(c-1)})(n-c+1)\cdots(n-1)=\pi^*.\] Letting \[\widehat\sigma=\widehat\mu^{(0)}\,n\,\widehat\mu^{(1)}(n-1)\,\widehat\mu^{(2)}\cdots\widehat\mu^{(c-1)}(n-c+1),\] we find that $\widehat\sigma\in\mathcal Z_c(\pi)$ and $\widehat\sigma^*=\lambda$. 

With $\sigma\in\mathcal Z_c(\pi)$ as in \eqref{Eq19}, we see that $n-1\in\Des(\sigma)$ if and only if $\mu^{(c-1)}$ is empty. According to \eqref{Eq20}, this occurs if and only if $\rmax(\sigma^*)=c-1$. Because $c-1<\tl(\pi)-1=\tl(\pi^*)=\tl(s(\sigma^*))$, Lemma \ref{Lem4} tells us that $\rmax(\sigma^*)=c-1$ if and only if $\zeil(\sigma^*)=c-1$. This shows that the map $\sigma\mapsto\sigma^*$ is a bijection from $\mathcal Z_c(\pi)\cap D_{n-1}$ to $\mathcal Z_{c-1}(\pi^*)$. 

For the rest of the proof, assume $\sigma\in\mathcal Z_c^{a,b}(\pi)$. If $\sigma\not\in D_{n-1}$, then $\des(\sigma^*)=\des(\sigma)=a$ and $\peak(\sigma^*)=\peak(\sigma)=b$. Furthermore, it follows from the preceding paragraph that $\zeil(\sigma^*)\geq c$. Hence, $\sigma^*\in\mathcal Z_{\geq c}^{a,b}(\pi^*)$. Conversely, if $\sigma^*\in\mathcal Z_{\geq c}^{a,b}(\pi^*)$, then $\sigma\not\in D_{n-1}$ because $\des(\pi)=\des(\pi^*)$. This completes the proof of the bijection $\mathcal Z_c^{a,b}(\pi)\setminus D_{n-1}\to\mathcal Z_{\geq c}^{a,b}(\pi^*)$. 

If $\sigma\in D_{n-1}\setminus D_{n-2}$ (meaning $n-1$ is a peak of $\sigma$), then we can easily check that $\des(\sigma^*)=\des(\sigma)-1=a-1$ and $\peak(\sigma^*)=\peak(\sigma)-1=b-1$. Furthermore, we know from above that $\zeil(\sigma^*)=c-1$. Hence, $\sigma^*\in\mathcal Z_{c-1}^{a-1,b-1}(\pi^*)$. Conversely, if $\sigma^*\in\mathcal Z_{c-1}^{a-1,b-1}(\pi^*)$, then $n-1$ must be a peak of $\sigma$ because $\peak(\sigma^*)<\peak(\sigma)$. This implies that $\sigma\in D_{n-1}\setminus D_{n-2}$. This completes the proof of the bijection $(\mathcal Z_c^{a,b}(\pi)\cap D_{n-1})\setminus D_{n-2}\to\mathcal Z_{c-1}^{a-1,b-1}(\pi^*)$. 

If $\sigma\in D_{n-1}\cap D_{n-2}$, then $\des(\sigma^*)=\des(\sigma)-1=a-1$ and $\peak(\sigma^*)=\peak(\sigma)=b$. Furthermore, we know from above that $\zeil(\sigma^*)=c-1$. Hence, $\sigma^*\in\mathcal Z_{c-1}^{a-1,b}(\pi^*)$. Conversely, if $\sigma^*\in\mathcal Z_{c-1}^{a-1,b}(\pi^*)$, then $n-1$ must be a descent of $\sigma$ and must not be a peak of $\sigma$. This implies that $\sigma\in D_{n-1}\cap D_{n-2}$. This completes the proof of the bijection $\mathcal Z_c^{a,b}(\pi)\cap D_{n-1}\cap D_{n-2}\to\mathcal Z_{c-1}^{a-1,b}(\pi^*)$.   
\end{proof}

\begin{theorem}\label{Thm15}
Let \[\mathcal B=\bigcup_{m\geq 0}\{\widetilde\chi_m(1),\widetilde\chi_m(21),\widetilde\chi_m(2431)\}\] be the set from Theorem \ref{Thm13}.
Let $\tau^{(1)},\tau^{(2)},\ldots$ be a (possibly empty) list of patterns taken from the set $\mathcal B$, and let $\tau'^{(i)}=\swl(\tau^{(i)})$ for all $i$. For every $n\geq 0$, the statistics $\des$, $\peak$, and $\zeil$ are jointly equidistributed on \[s^{-1}(\Av_n(132,3412,\tau^{(1)},\tau^{(2)},\ldots))\quad\text{and}\quad s^{-1}(\Av_n(312,1342,\tau'^{(1)},\tau'^{(2)},\ldots)).\]
\end{theorem} 

\begin{proof}
We saw in the proof of Theorem \ref{Thm13} that \[\swl(\Av_n(132,3412,\tau^{(1)},\tau^{(2)},\ldots))=\Av_n(312,1342,\tau'^{(1)},\tau'^{(2)},\ldots),\] so it suffices to prove that 
\begin{equation}\label{Eq9}
|\mathcal Z_c^{a,b}(\pi)|=|\mathcal Z_c^{a,b}(\swl(\pi))|
\end{equation} for all $a,b,c\in\mathbb Z$, and $\pi\in\Av_n(132,3412,\tau^{(1)},\tau^{(2)},\ldots)$. We do this by induction on $n$, noting that the result is trivial if $n\leq 2$ since $\pi=\swl(\pi)$ in that case. Assume $n\geq 3$, and fix $a,b,c\in\mathbb Z$ and $\pi\in\Av_n(132,3412,\tau^{(1)},\tau^{(2)},\ldots)$. We may assume $a,b\geq 0$ and $c\geq 1$ since both sides of \eqref{Eq9} are $0$ otherwise. To ease notation, let $\pi'=\swl(\pi)$.  

First, assume $c\leq\tl(\pi)-1$. We know by Lemma \ref{Lem1} that $\tl(\pi)=\tl(\pi')$, so $c\leq\tl(\pi')-1$. Lemma~\ref{Lem3} tells us that \[|\mathcal Z_c^{a,b}(\pi)|=|\mathcal Z_c^{a,b}(\pi)\setminus D_{n-1}|+|(\mathcal Z_c^{a,b}(\pi)\cap D_{n-1})\setminus D_{n-2}|+|\mathcal Z_c^{a,b}(\pi)\cap D_{n-1}\cap D_{n-2}|\] 
\begin{equation}\label{Eq23}
=|\mathcal Z_{\geq c}^{a,b}(\pi^*)|+|\mathcal Z_{c-1}^{a-1,b-1}(\pi^*)|+|\mathcal Z_{c-1}^{a-1,b}(\pi^*)|.
\end{equation}
Similarly, we can replace $\pi$ with $\pi'$ in Lemma \ref{Lem3} to obtain 
\begin{equation}\label{Eq24}
|\mathcal Z_c^{a,b}(\pi')|=|\mathcal Z_{\geq c}^{a,b}((\pi')^*)|+|\mathcal Z_{c-1}^{a-1,b-1}((\pi')^*)|+|\mathcal Z_{c-1}^{a-1,b}((\pi')^*)|.
\end{equation}
By induction on $n$, we know that $|\mathcal Z_{c'}^{a',b'}(\pi^*)|=|\mathcal Z_{c'}^{a',b'}(\swl(\pi^*))|$ for all $a',b',c'\in\mathbb Z$.  Hence,
\[|\mathcal Z_{\geq c}^{a,b}(\pi^*)|=|\mathcal Z_{\geq c}^{a,b}(\swl(\pi^*))|,
\quad |\mathcal Z_{c-1}^{a-1,b-1}(\pi^*)|=|\mathcal Z_{c-1}^{a-1,b-1}(\swl(\pi^*))|,\]
\begin{equation}\label{Eq25}
\text{and}\quad |\mathcal Z_{c-1}^{a-1,b}(\pi^*)|=|\mathcal Z_{c-1}^{a-1,b}(\swl(\pi^*))|.
\end{equation}
Because $\tl(\pi)\geq c+1\geq 2$, we know that $\pi=\pi^*n$. This implies that $\swl(\pi^*)=(\pi')^*$, so we can combine \eqref{Eq23}, \eqref{Eq24}, and \eqref{Eq25} to find that $|\mathcal Z_c^{a,b}(\pi)|=|\mathcal Z_c^{a,b}(\pi')|$, as desired. 

We now assume $c\geq\tl(\pi)$. Lemma \ref{Lem4} tells us that $\zeil(\sigma)\leq\tl(\pi)$ and $\zeil(\sigma')\leq\tl(\pi')=\tl(\pi)$ for all $\sigma\in s^{-1}(\pi)$ and $\sigma'\in s^{-1}(\pi')$. Consequently, 
\begin{equation}\label{Eq28}
\mathcal Z_i(\pi)=\mathcal Z_i(\pi')=\emptyset\quad\text{for all}\quad i>\tl(\pi).
\end{equation} This shows that $|\mathcal Z_c^{a,b}(\pi)|=|\mathcal Z_c^{a,b}(\pi')|=0$ if $c>\tl(\pi)$. We are left to handle the case $c=\tl(\pi)$. Using \eqref{Eq28} and the fact that $\tl(\pi)=\tl(\pi')$, we deduce that \[|\mathcal Z_{\geq 1}^{a,b}(\pi)|=|\mathcal Z_{\tl(\pi)}^{a,b}(\pi)|+\sum_{i=1}^{\tl(\pi)-1}|\mathcal Z_i^{a,b}(\pi)|\quad \text{and}\quad|\mathcal Z_{\geq 1}^{a,b}(\pi')|=|\mathcal Z_{\tl(\pi)}^{a,b}(\pi')|+\sum_{i=1}^{\tl(\pi)-1}|\mathcal Z_i^{a,b}(\pi')|.\] 
We saw above that $|\mathcal Z_i^{a,b}(\pi)|=|\mathcal Z_i^{a,b}(\pi')|$ for all $i\in\{1,\ldots,\tl(\pi)-1\}$. Therefore, in order to prove that $|\mathcal Z_{\tl(\pi)}^{a,b}(\pi)|=|\mathcal Z_{\tl(\pi)}^{a,b}(\pi')|$, it suffices to prove that $|\mathcal Z_{\geq 1}^{a,b}(\pi)|=|\mathcal Z_{\geq 1}^{a,b}(\pi')|$. At this point, we use Proposition \ref{Prop4} to see that there is a type-preserving bijection $\theta_\pi:\VHC(\pi)\to\VHC(\pi')$. Mimicking the proof of Proposition \ref{Prop5}, we find that 
\begin{equation}\label{Eq29}
\sum_{\sigma\in s^{-1}(\pi)}x^{\des(\sigma)+1}y^{\peak(\sigma)+1}=\sum_{\sigma'\in s^{-1}(\pi')}x^{\des(\sigma')+1}y^{\peak(\sigma')+1}.
\end{equation} The coefficient of $x^{a+1}y^{b+1}$ in the polynomial on the left-hand side of \eqref{Eq29} is $|\mathcal Z_{\geq 1}^{a,b}(\pi)|$. The coefficient of $x^{a+1}y^{b+1}$ in the polynomial on the right-hand side of \eqref{Eq29} is $|\mathcal Z_{\geq 1}^{a,b}(\pi')|$. 
\end{proof}

\begin{remark}
In the specific case in which the sequence $\tau^{(1)},\tau^{(2)},\ldots$ consists of the single pattern $\widetilde\chi_2(1)=231$, Theorem \ref{Thm15} generalizes Corollary \ref{Cor2}. \hspace*{\fill}$\lozenge$  
\end{remark}

\section{Conclusion and Ideas for Future Work}\label{Sec:Conclusion}

When it comes to studying the stack-sorting map, valid hook configurations provide a unified framework for reproving and generalizing many known results and for discovering and proving completely new results. This theme is supported by all of the articles that make use of valid hook configurations \cite{DefantCatalan, DefantCounting, DefantDescents, DefantFertility, DefantPostorder, DefantPreimages, DefantClass, DefantEngenMiller, DefantKravitz}, and the current article is no exception. 

Consider the chain of equalities \[\sum_{n\geq 1}|s^{-1}(\Av_n(132,312))|x^n=\sum_{n\geq 1}|s^{-1}(\Av_n(231,312))|x^n=\sum_{n\geq 1}|s^{-1}(\Av_n(132,231))|x^n\] 
\begin{equation}\label{Eq1}
=\frac{1-2x-\sqrt{1-4x-4x^2}}{4x}.
\end{equation} The first of these equalities was proven in \cite{DefantClass}, and the other two equalities were conjectured there. The third was proven more recently in \cite{DefantEnumeration}. We completed this chain by proving the second equality as a consequence of Corollary \ref{Cor2} above. It is worth mentioning that the sequence with generating function $\displaystyle \frac{1-2x-\sqrt{1-4x-4x^2}}{4x}$ has appeared in other contexts. The reader can find more information about this sequence in \cite{Hossain}, where its terms were named ``Boolean-Catalan numbers." 

We can enlarge the sets counted in \eqref{Eq1} in order to obtain another chain of equalities. More precisely, it follows from Theorems \ref{Thm12} and \ref{Thm13} that \[\sum_{n\geq 1}|s^{-1}(\Av_n(132,3412))|x^n=\sum_{n\geq 1}|s^{-1}(\Av_n(231,1423))|x^n=\sum_{n\geq 1}|s^{-1}(\Av_n(312,1342))|x^n.\] 
This leads naturally to the following problem. 
\begin{problem}\label{Prob1}
Enumerate $s^{-1}(\Av(132,3412))$. 
\end{problem}

Throughout this article, we have introduced four variants of Wilf equivalence: fertility Wilf equivalence (Definition \ref{Def3}), postorder Wilf equivalence (Definition \ref{Def4}), binary postorder Wilf equivalence (Remark \ref{Rem4}), and strong fertility Wilf equivalence (Definition \ref{Def10}). We have seen that all of these properties imply fertility Wilf equivalence. In the discussion immediately following Remark \ref{Rem5}, we saw that there are infinitely many permutation classes that are strongly fertility Wilf equivalent but not binary postorder Wilf equivalent (hence, not postorder Wilf equivalent). This shows that binary postorder Wilf equivalence and postorder Wilf equivalence are strictly stronger than fertility Wilf equivalence. We now show that strong fertility Wilf equivalence is also strictly stronger than fertility Wilf equivalence. 

Given permutations $\lambda^{(1)},\ldots,\lambda^{(r)}$, let $\mathscr C(\lambda^{(1)},\ldots,\lambda^{(r)})$ be the set of all normalized permutations that are contained in at least one of the permutations $\lambda^{(1)},\ldots,\lambda^{(r)}$. This is a permutation class because it is equal to $\Av(\tau^{(1)},\tau^{(2)},\ldots)$, where $\tau^{(1)},\tau^{(2)},\ldots$ is a list of all of the patterns that are not in $\mathscr C(\lambda^{(1)},\ldots,\lambda^{(r)})$. Let $\mathscr D=\mathscr C(24135)$ and $\mathscr D'=\mathscr C(32415,31425,21435,42135)$. It is straightforward to check that 
\begin{equation}\label{Eq30}
(|s^{-1}(\mathscr D\cap S_n)|)_{n\geq 1}=(|s^{-1}(\mathscr D'\cap S_n)|)_{n\geq 1}=1,2,6,10,4,0,0,0,0,\ldots.
\end{equation} This shows that $\mathscr D$ and $\mathscr D'$ are fertility Wilf equivalent. However, \[|\VHC(\mathscr D\cap S_5)|=1<4=|\VHC(\mathscr D'\cap S_5)|.\] By Remark \ref{Rem6}, this forbids the existence of a type-preserving bijection $\VHC(\mathscr D)\to\VHC(\mathscr D')$. Therefore, $\mathscr D$ and $\mathscr D'$ are not strongly fertility Wilf equivalent.  

We also know that postorder Wilf equivalence implies binary postorder Wilf equivalence. As promised in Remark \ref{Rem4}, we now give a simple example showing that binary postorder Wilf equivalence does not imply postorder Wilf equivalence. It is not difficult to show that \[P^{-1}(\Av(123))\cap\DPT^{(2)}=\{I^{-1}(\epsilon),I^{-1}(1),I^{-1}(12),I^{-1}(21),I^{-1}(231)\}\] \[=P^{-1}(\Av(123,3214))\cap\DPT^{(2)},\] where $\epsilon$ is the empty permutation. Therefore, $\Av(123)$ and $\Av(123,3214)$ are binary postorder Wilf equivalent. There exists a decreasing ternary plane tree (ternary means that every vertex has exactly $3$ (possibly empty) subtrees) with postorder $3214$; this tree is an element of $P^{-1}(\Av(123))$ with $4$ vertices. However, one can check that there are no trees with $4$ vertices in $P^{-1}(\Av(123,3214))$. Hence, $\Av(123)$ and $\Av(123,3214)$ are not postorder Wilf equivalent.   

By the discussion following Remark \ref{Rem5}, strong fertility Wilf equivalence does not imply binary postorder Wilf equivalence or postorder Wilf equivalence. However, we do not know about the reverse implications. 

\begin{question}\label{Quest1}
Does there exist a pair of permutation classes that are binary postorder Wilf equivalent but are not strongly fertility Wilf equivalent?
\end{question} 

\begin{question}\label{Quest2}
Does there exist a pair of permutation classes that are postorder Wilf equivalent but are not strongly fertility Wilf equivalent?
\end{question} 

\section{Acknowledgments}
The author was supported by a Fannie and John Hertz Foundation Fellowship and an NSF Graduate Research Fellowship. He thanks the anonymous referees for helpful comments that improved the presentation of this article.

\end{document}